\documentclass[3p]{elsarticle}
\usepackage{amsmath}
\usepackage{amssymb}
\usepackage{float}
\usepackage{amsfonts}
\usepackage{mathrsfs}
\usepackage{tikz}
\usepackage{tikz-cd}
\usepackage{booktabs}
\usepackage{multirow}
\usepackage{siunitx}
\usepackage{doi}
\usepackage{listings}
\usepackage{alltt}
\usepackage{tabularx}
\usepackage[font=small,labelfont=bf, hypcap=false]{caption}
\usepackage{amsthm}
\usepackage{hyperref}
\usepackage{enumitem}
\usepackage{graphicx}
\usepackage[title]{appendix}
\usepackage{xcolor}
\usepackage{textcomp}
\usepackage{manyfoot}
\usepackage{algorithm}
\usepackage{algorithmicx}
\usepackage{algpseudocode}
\usepackage{bookmark}
\usepackage{anyfontsize}
\usepackage{rotating}
\usepackage[utf8]{inputenc}
\usepackage[english]{babel}
\usepackage{array}

\usepackage[utf8]{inputenc}
\usepackage{amsmath,amssymb,amsthm,mathtools}
\usepackage{graphicx}
\usepackage{booktabs}
\usepackage{enumitem}
\usepackage{microtype}
\usepackage{xcolor}
\usepackage[capitalise,nameinlink,noabbrev]{cleveref}

\newtheorem{theorem}{Theorem}[section]
\newtheorem{lemma}[theorem]{Lemma}
\newtheorem{proposition}[theorem]{Proposition}
\newtheorem{corollary}[theorem]{Corollary}
\theoremstyle{definition}
\newtheorem{definition}[theorem]{Definition}
\newtheorem{example}[theorem]{Example}
\theoremstyle{remark}
\newtheorem{remark}[theorem]{Remark}
\numberwithin{equation}{section}

\newcommand{\R}{\mathbb{R}}
\newcommand{\N}{\mathbb{N}}
\newcommand{\D}{D_{q,\omega}}
\newcommand{\wo}{\omega_0}
\newcommand{\Is}{I^{*}}
\newcommand{\qn}[1]{[#1]_q}
\newcommand{\cl}{\mathrm{cl}}
\newcommand{\SO}{\mathrm{SO}(3)}
\DeclareMathOperator{\sgn}{sgn}
\newcommand{\abs}[1]{\lvert #1\rvert}
\newcommand{\norm}[1]{\lVert #1\rVert}

\begin{document}
    \title{Space curves in Hahn calculus: rotation structure, a fundamental theorem, and $q$-deformed invariants at the fixed point}
    
        \author[Umut]{Umut Selvi}
        \author[Umut]{Derya Sağlam\corref{cor1}}
        \ead{derya.saglam@hbv.edu.tr} 
            \cortext[cor1]{Corresponding author}

    \address[Umut]{Department of Mathematics, Ankara Hacı Bayram Veli University, Ankara, 06900, Türkiye}
    
    \begin{abstract}
       We develop the differential geometry of space curves in Hahn's quantum calculus, where the derivative is the difference quotient along the contraction $\sigma(t)=qt+\omega$ with fixed point $\omega_0=\omega/(1-q)$. The Hahn Frenet frame is transported along every orbit of $\sigma$ by rotations through turning and torsion angles, which identifies it with the discrete Frenet frame of the inscribed orbit polygon. An embedded regular Hahn curve is determined, up to linear reparametrization about $\omega_0$, by its trace and the induced contraction. We prove a sharp planarity criterion and a fundamental theorem via an infinite product of rotations and the Jackson--N\"orlund integral. For smooth curves, expansions uniform in $q$ show that at $\omega_0$ the Hahn curvature and torsion carry the factors $[2]_q/2$ and $[3]_q/3$. These factors, a sampling effect of geometrically decaying steps, disappear under edge-length normalization, but their quotient persists in the limiting ratio of torsion and turning angles of every polygon inscribed in a smooth curve with monotonically converging vertices and edge lengths decaying with ratio~$q$. In uniformly sampled relaxation towards a stable node, the factors can become measurable biases. We characterize generalized Hahn helices by a two-step relation between turning and torsion angles; their Lancret ratio at $\omega_0$ is $\pm([2]_q/2)\cot\vartheta$, where $\vartheta$ is the helix angle. Unlike the classical case, a constant ratio of torsion and curvature does not characterize helices: curves with constant curvature and torsion are never generalized Hahn helices, and, near $\omega_0$, unit-speed helices with a given ratio are unique up to similarities and linear reparametrization.
    \end{abstract}

    \begin{keyword}
        Hahn difference operator; Quantum calculus; Space curves; Discrete Frenet frame; Fundamental theorem of curves; Lancret theorem \\
        \MSC[2020] 53A04; 39A13; 53A70
    \end{keyword}

    \maketitle
	
    \section{Introduction}\label{sec:intro}

The Hahn difference operator
\[
  \D f(t)=\frac{f(qt+\omega)-f(t)}{(q-1)t+\omega},\qquad 0<q<1,\ \omega\in\R,
\]
was introduced by Hahn~\cite{Hahn1949} in his work on orthogonal polynomials. It contains the Jackson $q$-difference operator~\cite{Jackson1908} as the case $\omega=0$, and it formally reduces to the forward difference operator for $q=1$ and to the classical derivative as $q\to1$, $\omega\to0$. Its systematic theory, including the Jackson--N\"orlund integral as a right inverse of $\D$, was developed by Annaby, Hamza and Aldwoah~\cite{Aldwoah2009,AnnabyHamzaAldwoah2012}, and it has since become a substantial part of quantum calculus~\cite{KacCheung2002}: linear Hahn difference equations~\cite{HamzaAhmed2013}, the Leibniz rule and Fubini's theorem~\cite{HamzaMakharesh2016}, the Hahn quantum variational calculus~\cite{MalinowskaTorres2010}, general quantum difference operators containing $\D$~\cite{HamzaSarhanShehataAldwoah2015}, fractional Hahn calculus and boundary value problems~\cite{BrikshavanaSitthiwirattham2017,Sitthiwirattham2016,PatanarapeelertReunsumritSitthiwirattham2024}, Appell-type polynomials~\cite{OzarslanMalyali2025} and Nevanlinna theory~\cite{WangLong2025}. The Hahn operator fits into Aldwoah's framework of generalized time scales~\cite{Aldwoah2009}; on each orbit of $\sigma(t)=qt+\omega$ it acts as a forward or a backward difference operator in the sense of the calculus on time scales~\cite{BohnerPeterson2001}.

Geometric applications are much scarcer. Curves parametrized by time scales were studied by Guseinov and \"Ozy\i lmaz~\cite{GuseinovOzyilmaz2001}, \"Ozy\i lmaz~\cite{Ozyilmaz2006}, Atmaca~\cite{Atmaca2010}, Seyyido\u{g}lu et al.~\cite{Seyyidoglu2011} and Atmaca and Akg\"uller~\cite{AtmacaAkguller2015}, and Frenet--Serret formulas for curves defined by the Jackson $q$-derivative, which corresponds to $\omega=0$ in the present setting, were considered by Altun and Y\"uce~\cite{AltunYuce2019}. On the other hand, the differential geometry of polygons is a mature subject~\cite{BobenkoSuris2008,CarrollHankinsKoseSterling2014,Sauer1970}: discrete curvature and torsion are defined through turning and torsion angles, the discrete Frenet frame is transported from vertex to vertex by rotations~\cite{HuLundgrenNiemi2011}, a polygon is reconstructed from its invariants by discrete Frenet equations~\cite{CarrollHankinsKoseSterling2014}, discrete curve flows are linked to integrable systems~\cite{InoguchiKajiwaraMatsuuraOhta2012}, and polygonal general helices have recently been characterized by finite-step relations between turning and torsion angles~\cite{YoonLeeLee2026}. Frames along discrete curves are also used in applications, for instance in the analysis of protein backbones~\cite{HuLundgrenNiemi2011} and in computer graphics~\cite{WangJuttlerZhengLiu2008}.

The Hahn setting lies between these two worlds. For $t\neq\wo$, where $\wo=\omega/(1-q)$ is the fixed point of $\sigma(t)=qt+\omega$, the operator involves only the two points $t$ and $\sigma(t)$. The orbit $t,\sigma(t),\sigma^2(t),\dots$ converges geometrically to $\wo$, the parameter steps being $\delta(\sigma^n(t))=q^n\delta(t)$ with $\delta(t)=\sigma(t)-t$; for instance, the impact times of a ball bouncing with coefficient of restitution $q$ form such an orbit, which accumulates at a finite time (Zeno behaviour)~\cite{JohanssonEgerstedtLygerosSastry1999}. Along each orbit a Hahn curve is therefore a polygon whose edges shrink geometrically, whereas at $\wo$ the operator is, by definition, the classical derivative. A Hahn curve is thus a one-parameter family of polygons glued together at a point at which the structure becomes infinitesimal. The aim of this paper is to develop a Frenet theory that uses both aspects and to isolate phenomena that belong to the Hahn setting itself rather than to its two limiting cases.

Our first result makes this picture precise at the level of frames. Let $\gamma\colon I\to\R^3$ be a Frenet-regular Hahn curve (\cref{def:frenetregular}) on an interval $I\ni\wo$, with Hahn unit tangent $T=\D\gamma/\abs{\D\gamma}$, Hahn Frenet frame $(T,N,B)$, Hahn curvature $\kappa=\D T\cdot N$, Hahn torsion $\tau=-\D B\cdot N$ and tangential component $\beta=\D T\cdot T$, and let $F$ be the matrix with rows $T$, $N$, $B$. For $t\neq\wo$ we introduce the turning angle $\theta(t)$ from $T(t)$ to $T(\sigma(t))$ and the torsion angle $\varphi(t)$ from $B(t)$ to $B(\sigma(t))$. In \cref{thm:A} we show that the frame is transported along every orbit by an explicit rotation,
\[
  F(\sigma(t))=X(\varphi(t))\,Z(\theta(t))\,F(t)\qquad(t\neq\wo),
\]
where $Z(\theta)$ and $X(\varphi)$ are the rotations about the third and the first coordinate axis, acting on the rows of $F$: the frame is first rotated through the turning angle about the binormal and then through the torsion angle about the new tangent. Consequently $\delta\kappa=\sin\theta$ and $\delta\tau=\cos\theta\sin\varphi$, and the condition $\delta^2(\kappa^2+\tau^2)<1$, which is needed to recover the frame from the curvature and the torsion, holds if and only if $\cos\theta\cos\varphi\neq0$ (\cref{cor:relations}).

The parametrization of a regular Hahn curve is essentially rigid. The only reparametrizations compatible with the Hahn structure are the linear changes of parameter about $\wo$ (\cref{prop:reparam}), and a regular Hahn curve that is a topological embedding is determined, up to such a change, by its trace together with the contraction that $\sigma$ induces on it (\cref{rem:coordfree}); conversely, every sufficiently smooth increasing contraction of an embedded arc towards a point at which its derivative equals $q$ arises in this way, the Hahn parameter being its Koenigs coordinate (\cref{prop:coordfree}). In particular, unit Hahn speed is a genuine restriction rather than a normalization. \Cref{prop:coordfree} also shows that Hahn curves describe discrete-time dynamics along curves: if a map of $\R^3$ leaves an arc through a fixed point invariant and restricts to such a contraction on it, the iterates of each point of the arc are the vertices of an orbit polygon of the corresponding Hahn curve (\cref{sec:conclusion}). For instance, if a trajectory of an analytic flow converges to a non-resonant stable node tangentially to the slowest eigendirection, then its samples at uniform time steps are, from some index on, the vertices of an orbit polygon of a Hahn curve in Jackson's calculus whose parameter is the linearizing coordinate of the slowest mode (\cref{ex:relaxation}).

By \cref{thm:A}, along each orbit the Hahn Frenet frame is, up to the indexing convention, the discrete Frenet frame of the orbit polygon~\cite{CarrollHankinsKoseSterling2014,HuLundgrenNiemi2011}, and we use this connection throughout. The contribution of the paper does not lie in the algebra along a single orbit, which is that of polygons and of curves parametrized by time scales (\cref{rem:polygons}), but in the features that are specific to the Hahn setting: the geometric decay of the parameter steps, the gluing of all orbits at $\wo$ together with the continuity of the frame on $I$, and the normalization of the invariants by the parameter step.

These features govern the existence and uniqueness theory. In \cref{thm:B,thm:Bangle} we prove a fundamental theorem: continuous functions $\kappa>0$ and $\tau$ with $\delta^2(\kappa^2+\tau^2)<1$ or, without any such restriction, suitable continuous turning and torsion angles determine a unit-speed Hahn curve uniquely up to a proper rigid motion. The frame is obtained as an infinite product of rotations along the orbits and the curve as a Jackson--N\"orlund integral, and the natural initial condition is imposed at the fixed point, where the Hahn Frenet equations reduce to the classical ones (\cref{prop:fixedpoint}). Uniqueness rests on the continuity of the frame, which couples the different orbits (\cref{rem:orbits,rem:speed}). We also obtain a planarity criterion and show by an example that its non-degeneracy hypothesis cannot be dropped (\cref{thm:planar,rem:necessity}). As $q\to1$ with $\wo$ fixed, the curves given by \cref{thm:B} converge to the classical curve with the same curvature and torsion (\cref{prop:qto1}).

For a smooth curve sampled on the Hahn grid, the Hahn invariants can be compared with the classical ones. Let $\gamma$ be a regular $C^4$ curve with non-vanishing curvature, classical speed $s_{\cl}$, curvature $\kappa_{\cl}$ and torsion $\tau_{\cl}$, and let $\kappa_q$ and $\tau_q$ denote its Hahn curvature and torsion for the parameter $q\in(0,1)$, the fixed point $\wo$ being the same for all $q$. In \cref{thm:C} we prove first-order expansions that are uniform in $q$: on every compact interval containing $\wo$,
\[
  \kappa_q=\frac{\qn{2}}{2}\,s_{\cl}\kappa_{\cl}+O(\abs{\delta}),\quad
  \tau_q=\frac{\qn{3}}{3}\,s_{\cl}\tau_{\cl}+O(\abs{\delta}),\quad \qn{n}=1+q+\dots+q^{n-1},
\]
for $\abs{\delta(t)}$ sufficiently small, with constants that do not depend on $q$. These expansions contain the classical limit $q\to1$ (\cref{cor:classical}) as well as the limit at the fixed point for fixed $q$ (\cref{cor:fixedpoint}): at $\wo$ the Hahn curvature and torsion equal $s_{\cl}\kappa_{\cl}$ and $s_{\cl}\tau_{\cl}$ multiplied by $\qn{2}/2$ and $\qn{3}/3$, respectively. These factors arise because the Hahn derivative divides by the parameter step $\delta$, whereas the tangent and the binormal, which are built from two and from three consecutive points of an orbit, approximate the classical ones at the centroids of their parameters, and consecutive centroids are $\frac{\qn{2}}{2}\delta$ and $\frac{\qn{3}}{3}\delta$ apart (\cref{rem:qnumbers}). Dividing the angles by the mean length of two or three consecutive edges of the orbit polygons removes the factors (\cref{cor:edge}). Their quotient, however, is a property of the polygons themselves: for every polygon inscribed in such a curve, with vertices $\gamma(x_n)$ for a strictly monotone sequence $x_n\to x^*$ and edge lengths $\ell_n$ with bounded ratios, the ratio $\varphi_n/\theta_n$ of torsion and turning angles equals $\frac23(\ell_n+\ell_{n+1}+\ell_{n+2})/(\ell_n+\ell_{n+1})$ times $\tau_{\cl}(x^*)/\kappa_{\cl}(x^*)$ up to $O(\abs{x_n-x^*})$. It therefore tends to $2\qn{3}/(3\qn{2})$ times $\tau_{\cl}(x^*)/\kappa_{\cl}(x^*)$ if the edge lengths decay with ratio $q$, and to $\tau_{\cl}(x^*)/\kappa_{\cl}(x^*)$ for edges of asymptotically equal length (\cref{prop:polygons}). In \cref{ex:relaxation} these factors appear as biases of discrete estimates of curvature and torsion for a uniformly sampled trajectory of the Lorenz system near a stable node, where $q$ is determined by the time step and the slowest decay rate.

Finally, we study generalized Hahn helices, whose tangent makes a constant angle $\vartheta\in(0,\pi/2)$ with a fixed axis $u$. In \cref{thm:D} we prove an intrinsic Lancret-type theorem: a Frenet-regular Hahn curve is such a helix with $B\cdot u\neq0$ if and only if $\abs{\theta}<2\vartheta$ and an explicit relation between $\theta(t)$, $\theta(\sigma(t))$ and $\varphi(t)$ holds for all $t\neq\wo$, with a fixed sign $\varepsilon$. This relation involves only two consecutive steps of each orbit polygon and makes no reference to the axis. For each sign $\varepsilon$, the unit-speed helices of this kind with $\sgn(B\cdot u)=\varepsilon$ correspond bijectively, up to proper rigid motions, to their turning-angle functions (\cref{cor:allhelices}). The binormal of every such helix is differentiable at $\wo$, and the Lancret ratio there is $q$-deformed: $\tau(\wo)/\kappa(\wo)=\varepsilon\,\frac{\qn{2}}{2}\cot\vartheta$, and the ratio of the edge-normalized invariants of \cref{cor:edge} also has a $q$-deformed limit at $\wo$ (\cref{cor:lancretratio}). On each orbit, the characterization above is the counterpart of the conserved-vector characterization and of the phase formulation of polygonal general helices obtained recently by Yoon, Lee and Lee~\cite{YoonLeeLee2026}. What is specific to the Hahn setting is the $q$-deformation of the Lancret ratio, which does not occur for uniform steps~\cite[Proposition~5.5]{YoonLeeLee2026}, and the role of continuity: since the frame is continuous on $I$, the sign of $B\cdot u$ is constant and the axes of all orbits coincide, so that a single two-step relation suffices and the turning angles determine the helix, whereas for polygons given turning angles and helix angle in general admit several non-congruent helices (\cref{rem:polygonhelices}). Moreover, the classical characterization of general helices by a constant ratio $\tau/\kappa$ fails: Hahn curves with constant curvature and torsion are never generalized Hahn helices, and near $\wo$ the unit-speed generalized Hahn helices with a given constant ratio $\tau/\kappa$ are unique up to similarities and linear changes of parameter about $\wo$, their turning angle being determined by the Koenigs function of an explicit map (\cref{prop:constratio}).

The paper is organized as follows. \Cref{sec:prelim} collects the necessary facts on the Hahn operator and the Jackson--N\"orlund integral. \Cref{sec:curves} introduces Hahn curves and their Frenet apparatus, establishes the rotation structure and determines the admissible reparametrizations. \Cref{sec:fundamental} contains the fundamental theorem and the planarity criterion, \cref{sec:limits} the asymptotic results, and \cref{sec:helices} the theory of generalized Hahn helices. \Cref{sec:examples} presents high-precision numerical examples that illustrate the results, including uniformly sampled relaxation towards a stable node, and \cref{sec:conclusion} contains concluding remarks.

\section{Preliminaries on Hahn calculus}\label{sec:prelim}

Throughout, $q\in(0,1)$ and $\omega\in\R$ are fixed, and
\begin{gather*}
  \sigma(t)=qt+\omega,\qquad \delta(t)=\sigma(t)-t,\qquad \wo=\frac{\omega}{1-q},\\
  \qn{n}=\frac{1-q^n}{1-q}=1+q+\dots+q^{n-1}.
\end{gather*}
For $\omega>0$ this is the standard Hahn setting~\cite{AnnabyHamzaAldwoah2012,Hahn1949}; $\omega=0$ gives Jackson's $q$-calculus, with $\wo=0$.

\begin{lemma}\label{lem:sigma}
For all $t\in\R$ and $k\in\N_0$:
\begin{enumerate}[label=\textup{(\alph*)}]
\item $\sigma(t)-\wo=q(t-\wo)$, hence $\sigma^k(t)=\wo+q^k(t-\wo)$;
\item $\delta(t)=(1-q)(\wo-t)$; in particular $\delta(t)=0$ if and only if $t=\wo$, and $\sgn\delta(t)=\sgn(\wo-t)$;
\item $\delta(\sigma^k(t))=q^k\delta(t)$ and $\sum_{k=0}^{\infty}\abs{\delta(\sigma^k(t))}=\abs{\wo-t}$.
\end{enumerate}
\end{lemma}

\begin{proof}
(a) $qt+\omega-\omega/(1-q)=qt-q\omega/(1-q)=q(t-\wo)$; iterate. (b) $qt+\omega-t=(1-q)\bigl(\omega/(1-q)-t\bigr)$. (c) By (a) and (b), $\delta(\sigma^k(t))=(1-q)(\wo-\sigma^k(t))=(1-q)q^k(\wo-t)=q^k\delta(t)$, and $\sum_k q^k\abs{\delta(t)}=\abs{\delta(t)}/(1-q)=\abs{\wo-t}$.
\end{proof}

\emph{Standing assumption.} $I\subset\R$ is an interval of positive length with $\wo\in I$, and $\Is:=I\setminus\{\wo\}$. By \cref{lem:sigma}(a), $\sigma(t)$ lies between $t$ and $\wo$; hence $\sigma(I)\subset I$, and $\sigma^k(t)\to\wo$ monotonically for every $t\in I$.

\begin{definition}\label{def:hahnderivative}
Let $f\colon I\to\R^n$. The \emph{Hahn derivative} of $f$ is
\[
  \D f(t)=\frac{f(\sigma(t))-f(t)}{\delta(t)}\quad(t\in\Is),\qquad \D f(\wo)=f'(\wo),
\]
the latter provided that $f$ is differentiable at $\wo$; in this case $f$ is called \emph{Hahn differentiable} on $I$. We abbreviate $D=\D$.
\end{definition}

\begin{lemma}[{cf.~\cite{AnnabyHamzaAldwoah2012}}]\label{lem:rules}
Let $f,g\colon I\to\R^3$ and $\lambda\colon I\to\R$ be Hahn differentiable on $I$. Then, for every $t\in I$,
\begin{enumerate}[label=\textup{(\roman*)}]
\item $f(\sigma(t))=f(t)+\delta(t)\,Df(t)$;
\item $D(\lambda f)(t)=D\lambda(t)\,f(t)+\lambda(\sigma(t))\,Df(t)$, $D(f\cdot g)(t)=Df(t)\cdot g(t)+f(\sigma(t))\cdot Dg(t)$ and $D(f\times g)(t)=Df(t)\times g(t)+f(\sigma(t))\times Dg(t)$;
\item $Df\equiv0$ if $f$ is constant.
\end{enumerate}
\end{lemma}

\begin{proof}
For $t\in\Is$ the identities follow from the definition; for instance $f(\sigma(t))\cdot g(\sigma(t))-f(t)\cdot g(t)=f(\sigma(t))\cdot\bigl(g(\sigma(t))-g(t)\bigr)+\bigl(f(\sigma(t))-f(t)\bigr)\cdot g(t)$. At $t=\wo$ they are the classical rules, because $\sigma(\wo)=\wo$ and $\delta(\wo)=0$.
\end{proof}

For a function $f$ that is bounded near $\wo$ we use the Jackson--N\"orlund integral~\cite{AnnabyHamzaAldwoah2012,Hahn1949}
\begin{equation}\label{eq:JN}
  \int_{\wo}^{t}f(x)\,d_{q,\omega}x:=-\sum_{k=0}^{\infty}\delta(\sigma^k(t))\,f(\sigma^k(t))
  =\bigl((1-q)t-\omega\bigr)\sum_{k=0}^{\infty}q^k f\bigl(tq^k+\omega\qn{k}\bigr);
\end{equation}
the two expressions agree by \cref{lem:sigma}, since $\sigma^k(t)=tq^k+\omega\qn{k}$ and $-\delta(t)=(1-q)t-\omega$.

\begin{lemma}\label{lem:JN}
Let $f\colon I\to\R^n$ be bounded on a neighbourhood of $\wo$ in $I$, and let $\Psi(t):=\int_{\wo}^{t}f(x)\,d_{q,\omega}x$. Then the series~\eqref{eq:JN} converges absolutely for every $t\in I$, and:
\begin{enumerate}[label=\textup{(\alph*)}]
\item $\Psi(\wo)=0$, $\Psi$ is continuous at $\wo$, and $\Psi(\sigma(t))-\Psi(t)=\delta(t)f(t)$ for all $t\in I$;
\item if $f$ is continuous at $\wo$, then $\Psi$ is differentiable at $\wo$ with $\Psi'(\wo)=f(\wo)$; hence $D\Psi=f$ on $I$;
\item if $G\colon I\to\R^n$ is continuous at $\wo$ and $G(\sigma(t))=G(t)$ for all $t\in\Is$, then $G$ is constant. Consequently, if $G$ is continuous at $\wo$ and $DG=f$ on $\Is$, then $G=G(\wo)+\Psi$.
\end{enumerate}
\end{lemma}

\begin{proof}
Choose $\eta>0$ and $\Lambda\ge0$ with $\abs{f}\le\Lambda$ on $\mathcal U:=I\cap(\wo-\eta,\wo+\eta)$. The set $\mathcal U$ is $\sigma$-invariant, since $\abs{\sigma(t)-\wo}=q\abs{t-\wo}$. For $t\in I$ there is $k_0$ with $\sigma^k(t)\in\mathcal U$ for $k\ge k_0$, and then $\abs{\delta(\sigma^k(t))f(\sigma^k(t))}\le\Lambda q^k\abs{\delta(t)}$; this proves absolute convergence.

(a) All terms vanish for $t=\wo$. For $t\in\mathcal U$, \cref{lem:sigma}(c) gives $\abs{\Psi(t)}\le\Lambda\sum_kq^k\abs{\delta(t)}=\Lambda\abs{t-\wo}$, hence continuity at $\wo$. Shifting the summation index, $\Psi(\sigma(t))=-\sum_{k\ge1}\delta(\sigma^k(t))f(\sigma^k(t))=\Psi(t)+\delta(t)f(t)$.

(b) For $t\in\mathcal U\setminus\{\wo\}$, by \cref{lem:sigma},
\[
  \frac{\Psi(t)-\Psi(\wo)}{t-\wo}=-\frac{\delta(t)}{t-\wo}\sum_{k=0}^{\infty}q^kf(\sigma^k(t))=(1-q)\sum_{k=0}^{\infty}q^kf(\sigma^k(t)).
\]
As $t\to\wo$, each term tends to $q^kf(\wo)$ and is bounded by $\Lambda q^k$; by dominated convergence the limit is $(1-q)\sum_kq^kf(\wo)=f(\wo)$. Together with (a) this gives $D\Psi=f$ on $I$.

(c) For $t\in I$ and $n\in\N$ we have $G(t)=G(\sigma^n(t))\to G(\wo)$. For the second statement apply this to $G-\Psi$, which is continuous at $\wo$ and satisfies $(G-\Psi)(\sigma(t))=(G-\Psi)(t)$ by (a).
\end{proof}

\begin{remark}\label{rem:orbits}
A relation such as $DG=f$ on $\Is$ only links the values of $G$ along each orbit $\{\sigma^n(t)\}$; it is continuity at $\wo$ that couples the different orbits (\cref{lem:JN}(c)). This mechanism is used repeatedly below.
\end{remark}

\section{Hahn space curves and their rotation structure}\label{sec:curves}

\begin{definition}\label{def:curve}
A \emph{Hahn curve} is a map $\gamma\colon I\to\R^3$ that is Hahn differentiable on $I$. Its \emph{Hahn velocity} is $v:=D\gamma$ and its \emph{Hahn speed} is $s:=\abs{v}$. The curve is \emph{regular} if $s>0$ on $I$ and \emph{unit-speed} if $s\equiv1$; for a regular Hahn curve the \emph{Hahn unit tangent} is $T:=v/s$.
\end{definition}

\begin{remark}\label{rem:polygons}
For $t\in\Is$, \cref{lem:rules}(i) gives $\gamma(\sigma(t))-\gamma(t)=\delta(t)v(t)$. Hence $T(t)$ is the unit vector along the chord joining $\gamma(t)$ and $\gamma(\sigma(t))$, oriented in the direction of increasing parameter, and this chord has length $\abs{\delta(t)}s(t)$. For $t_0\in\Is$ the points $\gamma(\sigma^n(t_0))$, $n\ge0$, are therefore the vertices of a polygon, the \emph{orbit polygon} of $t_0$, whose edge $[\gamma(\sigma^n(t_0)),\gamma(\sigma^{n+1}(t_0))]$ is parallel to $T(\sigma^n(t_0))$ and has length $q^n\abs{\delta(t_0)}s(\sigma^n(t_0))$. Since $\gamma$ is continuous at $\wo$, the vertices converge to $\gamma(\wo)$, and $T(\wo)=\gamma'(\wo)/\abs{\gamma'(\wo)}$ is the classical unit tangent. In general $T(t)$ is not the classical tangent of $\gamma$ at $t\neq\wo$ (see \cref{fig:circle}). For $t_0<\wo$, the points $\sigma^n(t_0)$ are right-scattered points of the time scale $\{\sigma^n(t_0):n\ge0\}\cup\{\wo\}$, at which $D$ is the delta derivative~\cite{BohnerPeterson2001}; thus $T(\sigma^n(t_0))$ is the direction of the forward tangent line, in the sense of Guseinov and \"Ozy\i lmaz~\cite{GuseinovOzyilmaz2001}, of the restriction of $\gamma$ to this time scale. For $t_0>\wo$, $D$ is the nabla derivative of the corresponding time scale. Curvatures of curves parametrized by time scales were studied in~\cite{AtmacaAkguller2015,Seyyidoglu2011}.
\end{remark}

\begin{definition}\label{def:frenetregular}
A regular Hahn curve is \emph{Frenet-regular} if
\begin{enumerate}[label=\textup{(F\arabic*)}]
\item\label{F1} $T$ is continuous on $I$ and differentiable at $\wo$, and
\item\label{F2} $T(t)\times DT(t)\neq0$ for every $t\in I$.
\end{enumerate}
For $t\in\Is$ we have $T(t)\times DT(t)=T(t)\times T(\sigma(t))/\delta(t)$, so \ref{F2} says that consecutive edges of the orbit polygons are never parallel; at $\wo$ it says that $T'(\wo)\neq0$. The continuity required in \ref{F1} is not implied by Hahn differentiability, which only links the values of a function along each orbit (\cref{rem:orbits}); it couples the orbit polygons through different base points, and without it the uniqueness part of \cref{thm:B} fails (\cref{rem:speed}(iii)).
\end{definition}

\begin{lemma}\label{lem:continuityDf}
Let $f\colon I\to\R^n$ be differentiable at $\wo$. Then $Df(t)\to f'(\wo)=Df(\wo)$ as $t\to\wo$. In particular, for a Frenet-regular curve $DT$ is continuous on $I$.
\end{lemma}

\begin{proof}
Write $f(x)=f(\wo)+(x-\wo)\bigl(f'(\wo)+\epsilon(x)\bigr)$ with $\epsilon(x)\to0$ as $x\to\wo$. Since $\sigma(t)-\wo=q(t-\wo)$ and $\delta(t)=-(1-q)(t-\wo)$,
\[
  Df(t)=f'(\wo)+\frac{\epsilon(t)-q\,\epsilon(\sigma(t))}{1-q}\longrightarrow f'(\wo)\qquad(t\to\wo).
\]
For a Frenet-regular curve, $DT(t)=\bigl(T(\sigma(t))-T(t)\bigr)/\delta(t)$ is continuous on $\Is$ because $T$ is continuous, and it is continuous at $\wo$ by the first part.
\end{proof}

\begin{definition}\label{def:frame}
Let $\gamma$ be Frenet-regular. The \emph{Hahn binormal} and the \emph{Hahn principal normal} are
\[
  B:=\frac{T\times DT}{\abs{T\times DT}},\qquad N:=B\times T,
\]
so that $(T,N,B)$ is a positively oriented orthonormal frame ($T\times N=B$) at every $t\in I$; by \cref{lem:continuityDf}, $N$ and $B$ are continuous on $I$. We denote by $F(t)\in\SO$ the matrix whose rows are $T(t),N(t),B(t)$. The \emph{Hahn curvature}, the \emph{Hahn torsion} and the \emph{tangential component} are
\[
  \kappa:=DT\cdot N,\qquad \tau:=-DB\cdot N,\qquad \beta:=DT\cdot T,
\]
and the \emph{Hahn Frenet matrix} $M=(m_{ij})$ is defined by $DF=MF$, that is, $m_{ij}=DF_i\cdot F_j$ with $(F_1,F_2,F_3)=(T,N,B)$. The quantities $\kappa$, $\beta$ and the first row of $M$ are defined on all of $I$; $\tau$ and the last two rows of $M$ are defined on $\Is$, and at $\wo$ whenever $B$ is differentiable there.

For $t\in\Is$ the vector $B(t)$ is a multiple of $T(t)\times T(\sigma(t))$; hence $T(\sigma(t))\in\operatorname{span}\{T(t),N(t)\}$ and $B(t)\in\operatorname{span}\{N(\sigma(t)),B(\sigma(t))\}$. We may therefore define the \emph{turning angle} $\theta(t)\in(-\pi,\pi]$ and the \emph{torsion angle} $\varphi(t)\in(-\pi,\pi]$ by
\begin{equation}\label{eq:angles}
\begin{aligned}
  \cos\theta(t)&=T(t)\cdot T(\sigma(t)), &\qquad \sin\theta(t)&=T(\sigma(t))\cdot N(t),\\
  \cos\varphi(t)&=B(t)\cdot B(\sigma(t)), &\qquad \sin\varphi(t)&=N(\sigma(t))\cdot B(t).
\end{aligned}
\end{equation}
\end{definition}

Let
\begin{gather*}
  Z(\theta)=\begin{pmatrix}\cos\theta&\sin\theta&0\\-\sin\theta&\cos\theta&0\\0&0&1\end{pmatrix},\qquad
  X(\varphi)=\begin{pmatrix}1&0&0\\0&\cos\varphi&\sin\varphi\\0&-\sin\varphi&\cos\varphi\end{pmatrix},\\
  R(\theta,\varphi):=X(\varphi)Z(\theta).
\end{gather*}
These matrices act on frame matrices from the left, i.e.\ on their rows: for instance, the rows of $Z(\theta)F$ are $\cos\theta\,T+\sin\theta\,N$, $-\sin\theta\,T+\cos\theta\,N$ and $B$.

\begin{theorem}\label{thm:A}
Let $\gamma$ be Frenet-regular and $t\in\Is$. Then
\[
  F(\sigma(t))=R(\theta(t),\varphi(t))\,F(t),
\]
that is,
\begin{equation}\label{eq:rotation}
\begin{aligned}
  T(\sigma(t))&=\cos\theta\,T+\sin\theta\,N,\\
  N(\sigma(t))&=-\sin\theta\cos\varphi\,T+\cos\theta\cos\varphi\,N+\sin\varphi\,B,\\
  B(\sigma(t))&=\sin\theta\sin\varphi\,T-\cos\theta\sin\varphi\,N+\cos\varphi\,B,
\end{aligned}
\end{equation}
where $T,N,B,\theta,\varphi$ on the right-hand sides are evaluated at $t$. Moreover $\theta(t)\notin\{0,\pi\}$, $\sgn\theta(t)=\sgn\delta(t)$, and $M(t)=\bigl(R(\theta,\varphi)-I\bigr)/\delta(t)$, i.e.
\begin{equation}\label{eq:M}
  M=\frac1\delta\begin{pmatrix}\cos\theta-1&\sin\theta&0\\-\sin\theta\cos\varphi&\cos\theta\cos\varphi-1&\sin\varphi\\ \sin\theta\sin\varphi&-\cos\theta\sin\varphi&\cos\varphi-1\end{pmatrix}.
\end{equation}
In particular,
\begin{equation}\label{eq:invariants}
  \delta\kappa=\sin\theta,\qquad 1+\delta\beta=\cos\theta,\qquad \delta\tau=\cos\theta\sin\varphi .
\end{equation}
\end{theorem}

\begin{proof}
Write $T_1:=T(\sigma(t))$, $N_1:=N(\sigma(t))$, $B_1:=B(\sigma(t))$ and omit the argument $t$. Since $DT=(T_1-T)/\delta$, we have $T\times DT=T\times T_1/\delta$, hence $B=\sgn(\delta)\,T\times T_1/\abs{T\times T_1}$ and
\[
  \sin\theta=T_1\cdot(B\times T)=B\cdot(T\times T_1)=\sgn(\delta)\,\abs{T\times T_1}.
\]
By \ref{F2}, $\sin\theta\neq0$; thus $\theta\notin\{0,\pi\}$ and $\sgn\theta=\sgn\delta$. As $T_1\in\operatorname{span}\{T,N\}$, the first line of~\eqref{eq:rotation} follows from~\eqref{eq:angles}. Put $N':=-\sin\theta\,T+\cos\theta\,N$. The rows of $Z(\theta)F$ are $T_1,N',B$; they form a positively oriented orthonormal frame, so that $T_1\times N'=B$, $N'\times B=T_1$ and $B\times T_1=N'$. Since $F(\sigma(t))$ is orthonormal, $B_1\perp T_1$, so $B_1=(B_1\cdot N')N'+(B_1\cdot B)B=-\mu N'+\cos\varphi\,B$, where $\mu:=-B_1\cdot N'$. Consequently
\[
  N_1=B_1\times T_1=-\mu\,N'\times T_1+\cos\varphi\,B\times T_1=\mu\,B+\cos\varphi\,N',
\]
and $\mu=N_1\cdot B=\sin\varphi$ by~\eqref{eq:angles}. Thus $N_1=\cos\varphi\,N'+\sin\varphi\,B$ and $B_1=-\sin\varphi\,N'+\cos\varphi\,B$, that is, $F(\sigma(t))=X(\varphi)Z(\theta)F(t)$; expanding $N'$ gives~\eqref{eq:rotation}. Finally $DF(t)=\bigl(F(\sigma(t))-F(t)\bigr)/\delta=\bigl((R-I)/\delta\bigr)F(t)$, and $F(t)$ is orthogonal, so $M=(R-I)/\delta$. The entries $m_{11}=\beta$, $m_{12}=\kappa$ and $m_{32}=-\tau$ of~\eqref{eq:M} give~\eqref{eq:invariants}.
\end{proof}

\begin{corollary}\label{cor:relations}
Let $\gamma$ be Frenet-regular. Then on $\Is$:
\begin{enumerate}[label=\textup{(\alph*)}]
\item $\kappa=\abs{\sin\theta}/\abs{\delta}>0$ and $(1+\delta\beta)^2+\delta^2\kappa^2=1$; moreover $\beta\neq0$ and $\sgn\beta=-\sgn\delta$;
\item $m_{13}=0$, $m_{21}=-\kappa\cos\varphi$, $m_{22}=(\cos\theta\cos\varphi-1)/\delta$, $m_{23}=\sin\varphi/\delta$, $m_{31}=\kappa\sin\varphi$, $m_{33}=(\cos\varphi-1)/\delta$;
\item $1-\delta^2(\kappa^2+\tau^2)=\cos^2\theta\cos^2\varphi$; hence $\delta^2(\kappa^2+\tau^2)<1$ if and only if $\cos\theta\cos\varphi\neq0$;
\item if $\delta^2(\kappa^2+\tau^2)<1$ on $\Is$, then $\abs{\theta}<\pi/2$ and $\abs{\varphi}<\pi/2$ on $\Is$, and with $\varrho_1:=\sqrt{1-\delta^2\kappa^2}$ and $\varrho_2:=\sqrt{1-\delta^2\tau^2/\varrho_1^2}$ we have $\cos\theta=\varrho_1$, $\sin\varphi=\delta\tau/\varrho_1$, $\cos\varphi=\varrho_2$ and
\begin{gather*}
  \beta=\frac{\varrho_1-1}{\delta},\quad m_{21}=-\kappa\varrho_2,\quad m_{22}=\frac{\varrho_1\varrho_2-1}{\delta},\\
  m_{23}=\frac{\tau}{\varrho_1},\quad m_{31}=\frac{\delta\kappa\tau}{\varrho_1},\quad m_{33}=\frac{\varrho_2-1}{\delta}.
\end{gather*}
\end{enumerate}
\end{corollary}

\begin{proof}
(a) and (b) are read off from~\eqref{eq:M}--\eqref{eq:invariants} and \cref{thm:A}; note that $\cos\theta\neq1$ because $\sin\theta\neq0$. (c) $\delta^2\kappa^2+\delta^2\tau^2=\sin^2\theta+\cos^2\theta\sin^2\varphi=1-\cos^2\theta\cos^2\varphi$. (d) The functions $\cos\theta(t)=T(t)\cdot T(\sigma(t))$ and $\cos\varphi(t)=B(t)\cdot B(\sigma(t))$ are continuous on $\Is$, because $T$ and $B$ are continuous on $I$, and they tend to $1$ as $t\to\wo$, because $T$ and $B$ are continuous at $\wo$. By (c) they do not vanish on $\Is$, so they are positive on each of the (at most two) components of $\Is$. The formulas follow from~\eqref{eq:invariants} and (b).
\end{proof}

Classically, $N'\cdot B=-B'\cdot N$. In the Hahn setting the two candidates for the torsion differ: by~\eqref{eq:invariants} and \cref{cor:relations}(b), $m_{23}=DN\cdot B=\sin\varphi/\delta$, whereas $\tau=-DB\cdot N=\cos\theta\sin\varphi/\delta$. They coincide at $\wo$ whenever $B$ is differentiable there (\cref{prop:fixedpoint}(b)), so none of the results at the fixed point depends on the choice. We use $\tau=-DB\cdot N$, which mirrors $B'=-\tau N$ and yields the identity in \cref{cor:relations}(c); with $m_{23}$ instead, the additional hypothesis in \cref{thm:planar}(v) would disappear, and the condition in \cref{thm:B} would become $\delta^2\kappa^2<1$ and $\delta^2m_{23}^2<1$.

\begin{proposition}\label{prop:fixedpoint}
Let $\gamma$ be Frenet-regular.
\begin{enumerate}[label=\textup{(\alph*)}]
\item $T'(\wo)=\kappa(\wo)N(\wo)$ with $\kappa(\wo)=\abs{T'(\wo)}>0$, and $\beta(\wo)=0$; moreover, as $t\to\wo$,
\[
  \kappa(t)\to\kappa(\wo),\qquad \beta(t)\to0,\qquad \theta(t)\to0,\qquad \frac{\theta(t)}{\delta(t)}\to\kappa(\wo).
\]
\item If, in addition, $B$ is differentiable at $\wo$, then
\[
  M(\wo)=\begin{pmatrix}0&\kappa(\wo)&0\\-\kappa(\wo)&0&\tau(\wo)\\0&-\tau(\wo)&0\end{pmatrix},
\]
so the Hahn Frenet equations at $\wo$ are the classical Frenet equations of $F$, and, as $t\to\wo$,
\[
  \tau(t)\to\tau(\wo),\qquad \varphi(t)\to0,\qquad \frac{\varphi(t)}{\delta(t)}\to\tau(\wo).
\]
\end{enumerate}
\end{proposition}

\begin{proof}
(a) Since $\abs{T}\equiv1$, $T'(\wo)\perp T(\wo)$, and $T'(\wo)\perp B(\wo)$ because $B(\wo)$ is a multiple of $T(\wo)\times T'(\wo)$; hence $T'(\wo)=\kappa(\wo)N(\wo)$, $\beta(\wo)=0$, and $\kappa(\wo)=\abs{T'(\wo)}>0$ by \ref{F2}. By \cref{lem:continuityDf}, $DT(t)\to T'(\wo)$, while $T(t)\to T(\wo)$ and $N(t)\to N(\wo)$; hence $\kappa(t)\to\kappa(\wo)$ and $\beta(t)\to0$. Since $\cos\theta(t)=T(t)\cdot T(\sigma(t))\to1$, we have $\theta(t)\to0$, and eventually $\abs{\theta}\le\pi/2$, so that $\abs{\theta}\le\frac\pi2\abs{\sin\theta}=\frac\pi2\abs{\delta}\kappa$ by~\eqref{eq:invariants}. Thus $\theta/\delta$ is bounded, $(\theta-\sin\theta)/\delta=O(\theta^2)\to0$, and $\theta/\delta\to\kappa(\wo)$.

(b) Now $DF(\wo)=F'(\wo)$ ($N=B\times T$ is differentiable at $\wo$ as well). Differentiating $FF^{\top}=I$ shows that $M(\wo)=F'(\wo)F(\wo)^{\top}$ is skew-symmetric; together with $m_{13}(\wo)=T'(\wo)\cdot B(\wo)=0$, $m_{12}=\kappa(\wo)$ and $m_{32}=-\tau(\wo)$ this gives the form of $M(\wo)$. By \cref{lem:continuityDf}, $DB(t)\to B'(\wo)$, hence $\tau(t)\to\tau(\wo)$. Since $\cos\varphi(t)=B(t)\cdot B(\sigma(t))\to1$, we have $\varphi(t)\to0$ and eventually $\abs{\varphi}\le\frac\pi2\abs{\sin\varphi}$; with $\sin\varphi=\delta\tau/\cos\theta$ from~\eqref{eq:invariants} it follows as in (a) that $\varphi/\delta\to\tau(\wo)$.
\end{proof}

\begin{remark}\label{rem:DFF}
By \cref{rem:polygons} and \cref{thm:A}, along every orbit polygon the Hahn Frenet frame is transported from one edge to the next by a rotation through the turning angle about the binormal, followed by a rotation through the torsion angle about the new tangent. Up to the indexing convention, it is thus the discrete Frenet frame of the orbit polygon in the sense of~\cite{CarrollHankinsKoseSterling2014,HuLundgrenNiemi2011}; here the binormal $B(t)$ belongs to the vertex $\gamma(\sigma(t))$ that is shared by the edges with directions $T(t)$ and $T(\sigma(t))$, the edges being ordered by increasing parameter. \Cref{thm:A} also shows that $\kappa$ and $\beta$ carry no independent information: $\delta\kappa$ and $1+\delta\beta$ are the sine and the cosine of the turning angle.
\end{remark}

Since the Hahn structure is given by the fixed map $\sigma$, only special changes of parameter are compatible with it. Let $\gamma\colon I\to\R^3$ be a regular Hahn curve. We call a monotone bijection $h\colon I'\to I$, where $I'$ is an interval containing $\wo$, an \emph{admissible reparametrization} of $\gamma$ if $h\circ\sigma=\sigma\circ h$ on $I'$ and $\gamma\circ h$ is again a regular Hahn curve. The first condition means that $h$ maps orbits of $\sigma$ onto orbits, so that $\gamma$ and $\gamma\circ h$ have the same orbit polygons (\cref{rem:polygons}); without it, a change of parameter changes the polygons themselves.

\begin{proposition}\label{prop:reparam}
Let $\gamma\colon I\to\R^3$ be a regular Hahn curve, and let $h\colon I'\to I$ be a monotone bijection, where $I'$ is an interval containing $\wo$, with $h\circ\sigma=\sigma\circ h$ on $I'$. Then $\gamma\circ h$ is Hahn differentiable on $I'$ if and only if
\begin{equation}\label{eq:scaling}
  h(t)=\wo+\lambda\,(t-\wo)\qquad(t\in I')
\end{equation}
for some $\lambda\neq0$. In this case $\tilde\gamma:=\gamma\circ h$ is a regular Hahn curve, it is Frenet-regular if $\gamma$ is, and
\begin{equation}\label{eq:transf}
\begin{gathered}
  \tilde s=\abs{\lambda}\,s\circ h,\qquad \tilde\kappa=\abs{\lambda}\,\kappa\circ h,\qquad \tilde\tau=\abs{\lambda}\,\tau\circ h,\qquad \tilde\beta=\lambda\,\beta\circ h,\\
  \tilde T=\sgn(\lambda)\,T\circ h,\qquad \tilde N=N\circ h,\qquad \tilde B=\sgn(\lambda)\,B\circ h,\\
  \tilde\theta=\sgn(\lambda)\,\theta\circ h,\qquad \tilde\varphi\equiv\sgn(\lambda)\,\varphi\circ h\pmod{2\pi}.
\end{gathered}
\end{equation}
In particular, the edge lengths $\ell(t):=\abs{\gamma(\sigma(t))-\gamma(t)}=\abs{\delta(t)}\,s(t)$ of the orbit polygons and the quotients
\[
  \frac{\kappa}{s}=\frac{\abs{\sin\theta}}{\ell},\qquad \frac{\tau}{s}=\sgn(\delta)\,\frac{\cos\theta\sin\varphi}{\ell}\qquad\text{on }\Is
\]
transform by composition with $h$.
\end{proposition}

\begin{proof}
Since $h$ is a monotone bijection between intervals, it is continuous. As $\sigma(h(\wo))=h(\sigma(\wo))=h(\wo)$ and $\wo$ is the only fixed point of $\sigma$, we have $h(\wo)=\wo$. For $t\in I'\setminus\{\wo\}$ put $g(t):=(h(t)-\wo)/(t-\wo)$; then $g$ has constant sign, and $g(\sigma(t))=q\,(h(t)-\wo)/\bigl(q\,(t-\wo)\bigr)=g(t)$ by \cref{lem:sigma}(a). Suppose that $\gamma\circ h$ is differentiable at $\wo$. Since $h(t)\to\wo$ as $t\to\wo$ and $\gamma'(\wo)\neq0$, the identity
\[
  \frac{\gamma(h(t))-\gamma(\wo)}{t-\wo}=g(t)\,\frac{\gamma(h(t))-\gamma(\wo)}{h(t)-\wo}
\]
shows that $g$ has a finite limit $\lambda$ at $\wo$. Hence $g(t)=g(\sigma^n(t))\to\lambda$ for every $t$, i.e.\ $g\equiv\lambda$, and $\lambda\neq0$ because $h$ is injective. Conversely, if $h$ is given by~\eqref{eq:scaling}, then $\gamma\circ h$ is differentiable at $\wo$ by the chain rule.

Now $\delta(h(t))=h(\sigma(t))-h(t)=\lambda\,\delta(t)$, so that $D\tilde\gamma=\lambda\,(D\gamma)\circ h$ on $I'\setminus\{\wo\}$, and also at $\wo$ by the chain rule. This gives $\tilde s=\abs{\lambda}\,s\circ h$ and $\tilde T=\sgn(\lambda)\,T\circ h$, and in the same way $D\tilde T=\abs{\lambda}\,(DT)\circ h$. Hence $\tilde T\times D\tilde T=\lambda\,(T\times DT)\circ h$, which yields the Frenet-regularity of $\tilde\gamma$, $\tilde B=\sgn(\lambda)\,B\circ h$ and $\tilde N=\tilde B\times\tilde T=N\circ h$; then $D\tilde B=\abs{\lambda}\,(DB)\circ h$, and the remaining formulas in~\eqref{eq:transf} follow from \cref{def:frame}. Moreover $\tilde\ell=\ell\circ h$, because $h\circ\sigma=\sigma\circ h$. Finally, the expressions for $\kappa/s$ and $\tau/s$ follow from~\eqref{eq:invariants}, \cref{cor:relations}(a) and $\delta s=\sgn(\delta)\,\ell$.
\end{proof}

Thus the parametrization of a Hahn curve is rigid: the only freedom is a linear change of parameter about $\wo$, and for orientation-preserving changes it is removed by the normalization $s(\wo)=\abs{\gamma'(\wo)}=1$. The next proposition describes the geometric object behind a Hahn curve without reference to a parameter.

\begin{proposition}\label{prop:coordfree}
Let $c\colon\mathcal X\to\R^3$ be a $C^1$ embedding parametrized by arc length on an interval $\mathcal X\ni0$, put $C:=c(\mathcal X)$ and $p:=c(0)$, and let $\Sigma\colon\mathcal X\to\mathcal X$ be of class $C^2$ with $\Sigma'>0$, $\Sigma(0)=0$, $\Sigma'(0)=q$ and $0<\Sigma(x)/x<1$ for $x\neq0$. Then the limit $\mathcal K:=\lim_{n\to\infty}q^{-n}\Sigma^n$ exists and is an increasing $C^1$-diffeomorphism of $\mathcal X$ onto an interval $\mathcal J\ni0$ with $\mathcal K\circ\Sigma=q\,\mathcal K$ and $\mathcal K'(0)=1$, and
\[
  \gamma(t):=c\bigl(\mathcal K^{-1}(t-\wo)\bigr),\qquad t\in I:=\wo+\mathcal J,
\]
is an injective regular Hahn curve of class $C^1$ with $\gamma(I)=C$, $\gamma(\wo)=p$, $\abs{\gamma'(\wo)}=1$ and
\begin{equation}\label{eq:conj}
  \gamma\circ\sigma=c\circ\Sigma\circ c^{-1}\circ\gamma .
\end{equation}
Every continuous injective regular Hahn curve $\gamma_1\colon I_1\to\R^3$ with $\gamma_1(I_1)=C$ that satisfies~\eqref{eq:conj} is of the form $\gamma_1=\gamma\circ h$ with an admissible reparametrization $h$. Conversely, if $I\ni\wo$ is an interval and $\gamma\colon I\to\R^3$ is a $C^2$ embedding, i.e.\ an injective $C^2$ curve with $\gamma'\neq0$ that is a homeomorphism onto its image, then its arc-length parametrization $c$ with $c(0)=\gamma(\wo)$ and the map $\Sigma:=c^{-1}\circ\gamma\circ\sigma\circ\gamma^{-1}\circ c$ satisfy these assumptions, and~\eqref{eq:conj} holds.
\end{proposition}

\begin{proof}
Choose $\eta>0$ and $C_\Sigma>0$ with $C_\Sigma\eta<1$ such that $\abs{\Sigma(x)}\le q_1\abs{x}$, where $q_1:=(1+q)/2$, and $\abs{\Sigma'(x)/q-1}+\abs{\Sigma(x)/(qx)-1}\le C_\Sigma\abs{x}$ for $x\in\mathcal X$ with $0<\abs{x}\le\eta$; this is possible because $\Sigma$ is of class $C^2$ with $\Sigma(0)=0$ and $\Sigma'(0)=q$. For $x\in \mathcal X_\eta:=\mathcal X\cap[-\eta,\eta]$ the products
\[
  \frac{\Sigma^n(x)}{q^n}=x\prod_{k<n}\frac{\Sigma(\Sigma^k(x))}{q\,\Sigma^k(x)},\qquad \Bigl(\frac{\Sigma^n}{q^n}\Bigr)'(x)=\prod_{k<n}\frac{\Sigma'(\Sigma^k(x))}{q}
\]
converge uniformly, because their factors differ from $1$ by at most $C_\Sigma q_1^k\eta<1$. Hence $\mathcal K$ is of class $C^1$ on $\mathcal X_\eta$, with $\mathcal K'>0$, $\mathcal K'(0)=1$ and $\mathcal K\circ\Sigma=q\,\mathcal K$; this is Koenigs' linearization~\cite{Koenigs1884,KuczmaChoczewskiGer1990}. Since $0<\Sigma(x)/x<1$ for $x\neq0$, for every $x\in\mathcal X$ the numbers $\abs{\Sigma^m(x)}$ decrease to a limit $L$, the limit point of $\Sigma^m(x)$ is fixed by $\Sigma$ by continuity, and hence $L=0$; moreover $\mathcal K(x)=q^{-m}\mathcal K(\Sigma^m(x))$ for $m$ so large that $\abs{\Sigma^m(x)}\le\eta$; thus $\mathcal K$ is an increasing $C^1$-diffeomorphism of $\mathcal X$ onto an interval $\mathcal J\ni0$ with $\mathcal K\circ\Sigma=q\,\mathcal K$. By \cref{lem:sigma}(a), $\gamma(\sigma(t))=c\bigl(\mathcal K^{-1}(q(t-\wo))\bigr)=c\bigl(\Sigma(\mathcal K^{-1}(t-\wo))\bigr)$, which is~\eqref{eq:conj}. Moreover, $\gamma$ is of class $C^1$ with $\abs{\gamma'(\wo)}=1/\mathcal K'(0)=1$, and its Hahn speed is positive on $\Is$ because $\gamma$ is injective. If $\gamma_1$ is as in the statement, then $c^{-1}\circ\gamma_1$ is continuous and injective, hence monotone, so that $h:=\gamma^{-1}\circ\gamma_1=\wo+\mathcal K\circ c^{-1}\circ\gamma_1$ is a monotone bijection of $I_1$ onto $I$; it satisfies $h\circ\sigma=\sigma\circ h$ by~\eqref{eq:conj} and the injectivity of $\gamma$, and $\gamma\circ h=\gamma_1$ is Hahn differentiable. By \cref{prop:reparam}, $h$ is an admissible reparametrization.

For the converse, the arc-length function $\xi(t):=\int_{\wo}^{t}\abs{\gamma'(u)}\,du$ is a $C^2$-diffeomorphism of $I$ onto an interval $\mathcal X\ni0$, since $\xi'=\abs{\gamma'}$ is positive and of class $C^1$. Thus $c=\gamma\circ\xi^{-1}$ and $\Sigma=\xi\circ\sigma\circ\xi^{-1}$ are of class $C^2$, $\Sigma'>0$ and $\Sigma'(0)=\xi'(\wo)\,q\,(\xi^{-1})'(0)=q$. Finally, $0<\Sigma(x)/x<1$ for $x\neq0$ because $\sigma(t)$ lies strictly between $t$ and $\wo$ (\cref{lem:sigma}(a)), and~\eqref{eq:conj} holds by construction.
\end{proof}

\begin{remark}\label{rem:coordfree}
By \cref{prop:reparam,prop:coordfree}, under the smoothness assumptions of \cref{prop:coordfree} a continuous injective regular Hahn curve is, up to a linear change of parameter about $\wo$, the same as an arc together with a contraction along it, and the Hahn parameter is the Koenigs coordinate of the contraction. The uniqueness part needs no smoothness: if $\gamma_1\colon I_1\to\R^3$ and $\gamma_2\colon I_2\to\R^3$ are regular Hahn curves that are topological embeddings with the same trace and $\gamma_1\circ\sigma\circ\gamma_1^{-1}=\gamma_2\circ\sigma\circ\gamma_2^{-1}$, then $h:=\gamma_2^{-1}\circ\gamma_1$ is a homeomorphism of $I_1$ onto $I_2$ with $h\circ\sigma=\sigma\circ h$ and $\gamma_2\circ h=\gamma_1$, so that $h$ is a linear change of parameter about $\wo$ by \cref{prop:reparam}. Regularity cannot be dropped: for $\gamma(t)=((t-\wo)^3,0,0)$ and every increasing bijection $h$ with $h\circ\sigma=\sigma\circ h$, the curve $\gamma\circ h$ is again Hahn differentiable. The parametrization is thus part of the geometric data rather than a gauge freedom: the same arc with different contractions carries different orbit polygons and different Hahn invariants. In particular, a smooth curve sampled with two parametrizations that do not differ by a linear change of parameter about $\wo$ yields two different Hahn curves; this applies to the comparison with classical invariants in \cref{sec:limits}. For Frenet-regular curves normalized by $\abs{\gamma'(\wo)}=1$ and with a fixed orientation, the functions $s$, $\kappa$, $\tau$, $\beta$, $\theta$ and $\varphi$ are invariants of the arc together with the contraction under proper rigid motions. In terms of the orbit polygons, the parametrization-free data are the edge lengths $\ell$ and the angles $\theta$ and $\varphi$ along the orbits of the contraction, as functions of the Koenigs coordinate; by \cref{prop:reparam} they transform by composition with $h$, the angles changing sign when the orientation is reversed.
\end{remark}

\section{The fundamental theorem and the planarity criterion}\label{sec:fundamental}

We use the operator norm on $3\times3$ matrices.

\begin{lemma}\label{lem:products}
The following hold.
\begin{enumerate}[label=\textup{(\alph*)}]
\item For $\theta,\varphi\in\R$, $\norm{R(\theta,\varphi)-I}\le\abs{\theta}+\abs{\varphi}$.
\item Let $R_k\in\SO$, $k\ge0$, with $A:=\sum_k\norm{R_k-I}<\infty$. Then $\Pi_n:=R_{n-1}\cdots R_1R_0$ converges in $\SO$ to a limit $\Pi_\infty$, with $\norm{\Pi_\infty-\Pi_n}\le\sum_{k\ge n}\norm{R_k-I}$ and $\norm{\Pi_\infty-I-\sum_{k\ge0}(R_k-I)}\le A^2$.
\end{enumerate}
\end{lemma}

\begin{proof}
(a) $Z(\theta)-I$ is a normal matrix with eigenvalues $0$ and $e^{\pm i\theta}-1$, so $\norm{Z(\theta)-I}=2\abs{\sin(\theta/2)}\le\abs{\theta}$; similarly $\norm{X(\varphi)-I}\le\abs{\varphi}$. For orthogonal $U,V$ we have $\norm{UV-I}\le\norm{U(V-I)}+\norm{U-I}=\norm{V-I}+\norm{U-I}$.

(b) Put $a_k:=\norm{R_k-I}$. By the last inequality and induction, $\norm{R_m\cdots R_n-I}\le\sum_{k=n}^{m}a_k$; hence $\norm{\Pi_{m+1}-\Pi_n}=\norm{(R_m\cdots R_n-I)\Pi_n}\le\sum_{k=n}^{m}a_k$. Thus $(\Pi_n)$ is a Cauchy sequence, and its limit belongs to the closed set $\SO$. Finally, $E_n:=\Pi_n-I-\sum_{k<n}(R_k-I)$ satisfies $E_0=0$ and $E_{n+1}=E_n+(R_n-I)(\Pi_n-I)$; since $\norm{\Pi_n-I}\le A$, we get $\norm{E_n}\le A\sum_{k<n}a_k\le A^2$.
\end{proof}

\begin{theorem}[fundamental theorem]\label{thm:B}
Let $\kappa\colon I\to(0,\infty)$ and $\tau\colon I\to\R$ be continuous and assume
\begin{equation}\label{eq:bound}
  \delta(t)^2\bigl(\kappa(t)^2+\tau(t)^2\bigr)<1\qquad(t\in I).
\end{equation}
\begin{enumerate}[label=\textup{(\alph*)}]
\item For every $p\in\R^3$ and every $F_0\in\SO$ there is a unit-speed Frenet-regular Hahn curve $\gamma\colon I\to\R^3$ with $\gamma(\wo)=p$ and $F(\wo)=F_0$ whose Hahn curvature and Hahn torsion are $\kappa$ and $\tau$ on $I$. Its Hahn Frenet frame $F$ is continuous on $I$ and differentiable at $\wo$, with $F'(\wo)=M_0F_0$, where
\[
  M_0=\begin{pmatrix}0&\kappa(\wo)&0\\-\kappa(\wo)&0&\tau(\wo)\\0&-\tau(\wo)&0\end{pmatrix}.
\]
\item If $\tilde\gamma\colon I\to\R^3$ is a unit-speed Frenet-regular Hahn curve whose Hahn curvature and Hahn torsion coincide with $\kappa$ and $\tau$ on $\Is$, then $\tilde\gamma=U\gamma+b$ for some $U\in\SO$ and $b\in\R^3$.
\end{enumerate}
\end{theorem}

\begin{proof}
\emph{Step 1 (step rotations).} By~\eqref{eq:bound}, $\abs{\delta\kappa}<1$ and $\abs{\delta\tau}<\sqrt{1-\delta^2\kappa^2}$, so the functions
\[
  \theta(t):=\arcsin\bigl(\delta(t)\kappa(t)\bigr),\qquad \varphi(t):=\arcsin\Bigl(\frac{\delta(t)\tau(t)}{\cos\theta(t)}\Bigr)
\]
are well defined and continuous on $I$, take values in $(-\pi/2,\pi/2)$, satisfy $\theta(\wo)=\varphi(\wo)=0$, and $\sgn\theta=\sgn\delta$ on $\Is$ because $\kappa>0$. Let $R(t):=R(\theta(t),\varphi(t))$; then $R$ is continuous and $R(\wo)=I$.

\emph{Step 2 (estimates near $\wo$).} Choose $\eta>0$ and $C_1>0$ such that $\kappa\le C_1$, $\abs{\tau}\le C_1$ and $C_1\abs{\delta}\le\frac12$ on $\mathcal U:=I\cap(\wo-\eta,\wo+\eta)$. On $\mathcal U$ we then have $\cos\theta\ge\sqrt3/2$ and, since $\abs{\arcsin x}\le\frac\pi2\abs{x}$ for $\abs{x}\le1$,
\begin{equation}\label{eq:C0}
  \abs{\theta}+\abs{\varphi}\le\tfrac{\pi}{2}\bigl(1+2/\sqrt3\bigr)C_1\abs{\delta}=:C_0\abs{\delta},
\end{equation}
so that $\norm{R(t)-I}\le C_0\abs{\delta(t)}$ on $\mathcal U$ by \cref{lem:products}(a). For $t\in\mathcal U\setminus\{\wo\}$ put $M(t):=(R(t)-I)/\delta(t)$; then $\norm{M(t)}\le C_0$, and its entries are given by~\eqref{eq:M}. Since $\sin\theta/\delta=\kappa$, $\cos\theta\sin\varphi/\delta=\tau$, $\sin\varphi/\delta=\tau/\cos\theta$, while $(1-\cos\theta)/\delta$, $(1-\cos\varphi)/\delta$, $(1-\cos\theta\cos\varphi)/\delta$ and $\sin\theta\sin\varphi/\delta$ are $O(\abs{\delta})$ by~\eqref{eq:C0}, the continuity of $\kappa$ and $\tau$ at $\wo$ gives $M(t)\to M_0$ as $t\to\wo$.

\emph{Step 3 (the frame).} Fix $t\in I$. As $\sigma^k(t)\to\wo$, there is $k_0$ with $\sigma^k(t)\in\mathcal U$ for $k\ge k_0$, and then $\sum_{k\ge k_0}\norm{R(\sigma^k(t))-I}\le C_0\sum_kq^k\abs{\delta(t)}<\infty$ by \cref{lem:sigma}(c). By \cref{lem:products}(b), the limit
\[
  \Pi(t):=\lim_{n\to\infty}R(\sigma^{n-1}(t))\cdots R(\sigma(t))R(t)\in\SO
\]
exists. By construction,
\begin{equation}\label{eq:Pi}
  \begin{gathered}
  \Pi(t)=\Pi(\sigma(t))R(t)\quad(t\in I),\qquad \Pi(\wo)=I,\\
  \norm{\Pi(t)-I}\le C_0\abs{t-\wo}\quad(t\in\mathcal U),
  \end{gathered}
\end{equation}
the last estimate because $\sum_kC_0q^k\abs{\delta(t)}=C_0\abs{t-\wo}$. On a compact interval $L\subset I$ the index $k_0$ can be chosen uniformly, since $\abs{\sigma^k(t)-\wo}\le q^k\max_L\abs{t-\wo}$, so the tail estimate of \cref{lem:products}(b) is uniform on $L$; thus $\Pi$ is a uniform limit of continuous functions on $L$, and hence continuous on $I$. Define
\[
  F(t):=\Pi(t)^{\top}F_0\qquad(t\in I).
\]
Then $F(t)\in\SO$, i.e.\ its rows $T(t),N(t),B(t)$ form a positively oriented orthonormal frame, $F$ is continuous on $I$, $F(\wo)=F_0$, and~\eqref{eq:Pi} yields
\begin{equation}\label{eq:Fstep}
  F(\sigma(t))=R(t)F(t)\qquad(t\in I).
\end{equation}

\emph{Step 4 (differentiability at $\wo$).} Let $t\in\mathcal U\setminus\{\wo\}$. By \cref{lem:products}(b) and~\eqref{eq:Pi},
\[
  \Pi(t)=I+\sum_{k\ge0}\bigl(R(\sigma^k(t))-I\bigr)+E(t),\qquad \norm{E(t)}\le C_0^2\abs{t-\wo}^2,
\]
and $\sum_k\bigl(R(\sigma^k(t))-I\bigr)=\delta(t)\sum_kq^kM(\sigma^k(t))$. For each $k$, $M(\sigma^k(t))\to M_0$ as $t\to\wo$, and $\norm{q^kM(\sigma^k(t))}\le C_0q^k$; by dominated convergence $\sum_kq^kM(\sigma^k(t))\to M_0/(1-q)$. Since $\delta(t)/(1-q)=\wo-t$, we obtain $\Pi(t)=I+(\wo-t)M_0+o(\abs{t-\wo})$ and, $M_0$ being skew-symmetric,
\[
  F(t)=F_0+(t-\wo)M_0F_0+o(\abs{t-\wo}).
\]
Hence $F$ is differentiable at $\wo$ with $F'(\wo)=M_0F_0$.

\emph{Step 5 (the curve).} Let $\gamma(t):=p+\int_{\wo}^{t}T(x)\,d_{q,\omega}x$, where $T$ is the first row of $F$. By \cref{lem:JN}, $\gamma$ is continuous at $\wo$, $\gamma(\wo)=p$ and $D\gamma=T$ on $I$. Thus $\gamma$ is a unit-speed Hahn curve with Hahn unit tangent $T$, and \ref{F1} holds by Steps 3 and 4.

\emph{Step 6 (identification).} For $t\in\Is$ the first row of~\eqref{eq:Fstep} reads $T(\sigma(t))=\cos\theta\,T(t)+\sin\theta\,N(t)$; hence $T\times DT=(\sin\theta/\delta)\,T\times N=\kappa B$ with $\kappa>0$. At $\wo$, Step 4 gives $T'(\wo)=\kappa(\wo)N(\wo)$, so $T\times DT=\kappa(\wo)B$ there as well. Therefore \ref{F2} holds, $\gamma$ is Frenet-regular, and its Hahn Frenet frame is $F$. By \cref{thm:A} applied to $\gamma$, $F(\sigma(t))=R(\theta_\gamma(t),\varphi_\gamma(t))F(t)$ with the turning and torsion angles $\theta_\gamma,\varphi_\gamma$ of $\gamma$; comparing with~\eqref{eq:Fstep}, and observing that the entries $(1,1),(1,2),(2,3),(3,3)$ of $R(\theta,\varphi)$ determine $\theta$ and $\varphi$ modulo $2\pi$, we get $\theta_\gamma=\theta$ and $\varphi_\gamma=\varphi$ on $\Is$. Now~\eqref{eq:invariants} gives $\kappa_\gamma=\sin\theta/\delta=\kappa$ and $\tau_\gamma=\cos\theta\sin\varphi/\delta=\tau$ on $\Is$, and Step 4 gives $\kappa_\gamma(\wo)=T'(\wo)\cdot N(\wo)=\kappa(\wo)$ and $\tau_\gamma(\wo)=-B'(\wo)\cdot N(\wo)=\tau(\wo)$.

\emph{Step 7 (uniqueness).} Let $\tilde\gamma$ be as in (b), with frame $\tilde F$ and angles $\tilde\theta,\tilde\varphi$. By~\eqref{eq:bound} and \cref{cor:relations}(d), $\abs{\tilde\theta},\abs{\tilde\varphi}<\pi/2$, $\sin\tilde\theta=\delta\kappa$ and $\sin\tilde\varphi=\delta\tau/\sqrt{1-\delta^2\kappa^2}$; hence $\tilde\theta=\theta$ and $\tilde\varphi=\varphi$ on $\Is$, and \cref{thm:A} gives $\tilde F(\sigma(t))=R(t)\tilde F(t)$. Iterating, $\tilde F(\sigma^n(t))=R(\sigma^{n-1}(t))\cdots R(t)\tilde F(t)$, and letting $n\to\infty$ we obtain from the continuity of $\tilde F$ at $\wo$ that $\tilde F(\wo)=\Pi(t)\tilde F(t)$, i.e.\ $\tilde F(t)=\Pi(t)^{\top}\tilde F(\wo)$. Let $U\in\SO$ be the rotation that maps the frame $F_0$ onto $\tilde F(\wo)$, so that $\tilde F(\wo)=F_0U^{\top}$. Then $\tilde F(t)=F(t)U^{\top}$, i.e.\ $\tilde T=UT$ on $I$. The map $\tilde\gamma-U\gamma$ is continuous at $\wo$ and satisfies $D(\tilde\gamma-U\gamma)=\tilde T-UT=0$ on $\Is$, so it is constant by \cref{lem:JN}(c).
\end{proof}

\begin{remark}\label{rem:speed}
(i) If, in addition, a continuous function $s\colon I\to(0,\infty)$ is prescribed, then $\gamma(t)=p+\int_{\wo}^{t}s(x)T(x)\,d_{q,\omega}x$ is a regular Hahn curve with Hahn speed $s$ and with the same frame, curvature and torsion; thus $(s,\kappa,\tau)$ determine a Frenet-regular Hahn curve up to proper rigid motions. In contrast to the classical case, unit speed is a genuine restriction and not a normalization: by \cref{prop:reparam}, an admissible reparametrization changes $s$ into $\abs{\lambda}\,s\circ h$, so that a Frenet-regular Hahn curve with non-constant Hahn speed is not a reparametrization of a unit-speed one.
(ii) The initial condition at $\wo$ is natural: there the Hahn Frenet equations are the classical ones (\cref{prop:fixedpoint}), and the frame at any other point is obtained from $F(\wo)$ by the rotation $\Pi(t)^{\top}$. An initial frame at $t_0\in\Is$ may be prescribed instead, since $F(\wo)=\Pi(t_0)F(t_0)$.
(iii) Continuity of the frame at $\wo$ is essential for uniqueness: the relation $F(\sigma(t))=R(t)F(t)$ alone only links the values of $F$ along each orbit (\cref{rem:orbits}). Continuity on all of $I$ is needed as well, since it is used in \cref{cor:relations}(d) to recover $\cos\theta>0$ and $\cos\varphi>0$ from $(\kappa,\tau)$. Indeed, let $t_1\in\Is$ with $\sigma^{-1}(t_1)\notin I$ (such a point exists whenever $\sigma(I)\neq I$, for instance if $I$ is bounded), replace the torsion angle $\varphi(t_1)$ by $\pi-\varphi(t_1)$, i.e.\ $F(t_1)$ by $\tilde F(t_1):=R(\theta(t_1),\pi-\varphi(t_1))^{\top}F(\sigma(t_1))$, and $\gamma(t_1)$ by $\gamma(\sigma(t_1))-\delta(t_1)\tilde T(t_1)$, where $\tilde T(t_1)$ is the first row of $\tilde F(t_1)$. This changes neither $\kappa$ nor $\beta$ nor $\tau$ on $I$, but, since $\abs{\varphi(t_1)}<\pi/2$, it produces a unit-speed Hahn curve that is not congruent to $\gamma$ and whose tangent is discontinuous at $t_1$.
\end{remark}

Condition~\eqref{eq:bound} is exactly what is needed to recover the angles from $(\kappa,\tau)$ (\cref{cor:relations}). If the angles themselves are prescribed, no such condition is required.

\begin{theorem}\label{thm:Bangle}
Let $\theta,\varphi\colon I\to\R$ be continuous, with $\theta(\wo)=\varphi(\wo)=0$ and $0<\sgn(\delta(t))\,\theta(t)<\pi$ for $t\in\Is$. Assume that the limits $\kappa_0:=\lim_{t\to\wo}\theta(t)/\delta(t)>0$ and $\tau_0:=\lim_{t\to\wo}\varphi(t)/\delta(t)$ exist. Then for every $p\in\R^3$ and $F_0\in\SO$ there is a unit-speed Frenet-regular Hahn curve $\gamma$ with $\gamma(\wo)=p$, $F(\wo)=F_0$, turning angle $\theta$ and torsion angle congruent to $\varphi$ modulo $2\pi$; it is unique up to proper rigid motions. Its invariants are $\kappa=\sin\theta/\delta$, $\beta=(\cos\theta-1)/\delta$ and $\tau=\cos\theta\sin\varphi/\delta$ on $\Is$, and $\kappa(\wo)=\kappa_0$, $\tau(\wo)=\tau_0$.
\end{theorem}

\begin{proof}
Put $R(t):=R(\theta(t),\varphi(t))$. Since $\theta/\delta$ and $\varphi/\delta$ have finite limits at $\wo$, there are $\eta,C_0>0$ with $\abs{\theta}+\abs{\varphi}\le C_0\abs{\delta}$ on $\mathcal U=I\cap(\wo-\eta,\wo+\eta)$, and~\eqref{eq:M} shows that $(R(t)-I)/\delta(t)$ tends to the matrix $M_0$ built from $\kappa_0$ and $\tau_0$. Steps 3--6 of the proof of \cref{thm:B} now apply verbatim; in Step 6 the condition $\kappa>0$ is replaced by $\sgn\theta=\sgn\delta$, which gives $\sin\theta/\delta>0$, the comparison of the step rotations gives $\theta_\gamma=\theta$ and $\varphi_\gamma\equiv\varphi\pmod{2\pi}$, and the invariants follow from~\eqref{eq:invariants}. Two such curves with the same turning angle and with torsion angles that are congruent modulo $2\pi$ have the same step rotations $R(t)$, so Step 7 applies.
\end{proof}

The torsion angle~\eqref{eq:angles} of the curve of \cref{thm:Bangle} is the representative of $\varphi(t)$ modulo $2\pi$ in $(-\pi,\pi]$, so it coincides with $\varphi$ wherever $\abs{\varphi}<\pi$. No bound on $\varphi$ is needed, since $R(\theta,\varphi)$ depends on $\varphi$ only modulo $2\pi$. In particular, \cref{thm:Bangle} also yields curves whose torsion angle reaches $\pi$; where $\varphi$ crosses an odd multiple of $\pi$, the torsion angle~\eqref{eq:angles} jumps by $2\pi$, although the frame is continuous. Conversely, let $\gamma$ be a unit-speed Frenet-regular Hahn curve whose binormal is differentiable at $\wo$. By~\eqref{eq:angles}, $\cos\varphi$ and $\sin\varphi$ are continuous on $\Is$ and tend to $1$ and $0$ as $t\to\wo$, so that the torsion angle of $\gamma$ has a unique continuous lift $\tilde\varphi\colon I\to\R$ with $\tilde\varphi(\wo)=0$. By~\eqref{eq:angles}, \cref{thm:A} and \cref{prop:fixedpoint}, the turning angle of $\gamma$, extended by $0$ at $\wo$, is continuous, and together with $\tilde\varphi$ it satisfies the hypotheses of \cref{thm:Bangle}; by uniqueness, $\gamma$ is the curve that \cref{thm:Bangle} gives for $p=\gamma(\wo)$ and $F_0=F(\wo)$.

\begin{theorem}\label{thm:planar}
Let $\gamma$ be Frenet-regular. The following are equivalent:
\begin{enumerate}[label=\textup{(\roman*)}]
\item $\gamma(I)$ is contained in a plane;
\item $B$ is constant on $I$;
\item $\varphi\equiv0$ on $\Is$;
\item $m_{23}=DN\cdot B\equiv0$ on $\Is$.
\end{enumerate}
If, in addition, $1+\delta\beta\neq0$ on $\Is$ (equivalently, $\theta(t)\neq\pm\pi/2$), then \textup{(i)--(iv)} are also equivalent to
\begin{enumerate}[label=\textup{(\roman*)},start=5]
\item $\tau\equiv0$ on $\Is$.
\end{enumerate}
\end{theorem}

\begin{proof}
(i)$\Rightarrow$(ii) Let $\gamma(I)\subset\{x:x\cdot n=c_0\}$ with $\abs{n}=1$. Then $(\gamma\cdot n)(\sigma(t))=(\gamma\cdot n)(t)$, so $T(t)\cdot n=0$ for $t\in\Is$, and $T(\wo)\cdot n=\gamma'(\wo)\cdot n/s(\wo)=0$; differentiating $T\cdot n\equiv0$ at $\wo$ gives $T'(\wo)\cdot n=0$. Hence $T\times DT$ is parallel to $n$ on $I$, i.e.\ $B(t)\in\{n,-n\}$. As $B$ is continuous on the interval $I$, it is constant.
(ii)$\Rightarrow$(iii) $\cos\varphi=B(t)\cdot B(\sigma(t))=1$.
(iii)$\Rightarrow$(iv) By~\eqref{eq:M}, $m_{23}=\sin\varphi/\delta$.
(iv)$\Rightarrow$(ii) If $\sin\varphi\equiv0$, then $\cos\varphi(t)=B(t)\cdot B(\sigma(t))\in\{-1,1\}$ is continuous on each component of $\Is$ and tends to $1$ as $t\to\wo$, so $\varphi\equiv0$. By~\eqref{eq:rotation}, $B(\sigma(t))=B(t)$ on $\Is$, and \cref{lem:JN}(c) shows that $B$ is constant.
(ii)$\Rightarrow$(i) If $B\equiv n$, then $T\perp n$, so $(\gamma\cdot n)(\sigma(t))-(\gamma\cdot n)(t)=\delta(t)s(t)\,T(t)\cdot n=0$ on $\Is$, and $\gamma\cdot n$ is constant by \cref{lem:JN}(c).
Finally, $\delta\tau=\cos\theta\sin\varphi$ with $\cos\theta=1+\delta\beta$ by~\eqref{eq:invariants}; if $1+\delta\beta\neq0$, then $\tau\equiv0$ if and only if $\sin\varphi\equiv0$, i.e.\ (v)$\Leftrightarrow$(iv).
\end{proof}

\begin{remark}\label{rem:necessity}
Let $q=\frac12$, $\omega=1$ (so $\wo=2$ and $\delta(t)=1-t/2$), $I=[-3,2]$,
\begin{gather*}
  \theta(t):=\min\{\pi/2,\ \pi\delta(t)/2\},\qquad \varphi(t):=\tfrac12\,\chi(t+\tfrac32),\\
  \chi(x):=\begin{cases}e^{-1/(1-x^2)},&\abs{x}<1,\\0,&\abs{x}\ge1.\end{cases}
\end{gather*}
These data satisfy the hypotheses of \cref{thm:Bangle} with $\kappa_0=\pi/2$ and $\tau_0=0$. For the resulting curve, $\cos\theta=0$ wherever $\varphi\neq0$, because $\varphi$ is supported in $(-\frac52,-\frac12)$, where $\theta=\pi/2$. Hence $\tau=\cos\theta\sin\varphi/\delta\equiv0$ by~\eqref{eq:invariants}, although $\varphi\not\equiv0$, so the curve is not planar by \cref{thm:planar}. In fact, $B$ can be computed explicitly. Since $\sigma(I)=[-\frac12,2]$ and $\varphi$ vanishes outside $(-\frac52,-\frac12)$, we have $\varphi(\sigma^k(t))=0$ for $t\in I$ and $k\ge1$, so that $B(\sigma(t))=B(\wo)$ for $t\in I$ by~\eqref{eq:rotation} and the continuity of $B$ at $\wo$. As $F(t)=R(\theta(t),\varphi(t))^{\top}F(\sigma(t))$, it follows that $B(t)=\cos\varphi(t)\,B(\wo)+\sin\varphi(t)\,N(\sigma(t))$, hence $\abs{B(t)-B(\wo)}=2\abs{\sin(\varphi(t)/2)}$ for all $t\in I$; the largest value is $2\sin(1/(4e))\approx0.1837$, attained at $t=-\frac32$. Moreover, the planar curve with the same turning angle $\theta$ and with $\varphi\equiv0$ has the same Hahn curvature $\kappa=\sin\theta/\delta$ and the same torsion $\tau\equiv0$, but it is not congruent to the curve above, since its binormal is constant. Since $\delta^2\kappa^2=1$ wherever $\theta=\pi/2$, this shows that condition~\eqref{eq:bound} cannot be dropped from the uniqueness part of \cref{thm:B}.
\end{remark}

Finally, the curves of \cref{thm:B} converge to classical curves as $q\to1$.

\begin{proposition}\label{prop:qto1}
Let $J\subset\R$ be an open interval, $\wo\in J$, let $\kappa\colon J\to(0,\infty)$ and $\tau\colon J\to\R$ be continuous, let $p\in\R^3$ and $F_0\in\SO$, and let $\gamma_{\cl}\colon J\to\R^3$ be the curve parametrized by arc length with curvature $\kappa$, torsion $\tau$, $\gamma_{\cl}(\wo)=p$ and Frenet frame $F_{\cl}(\wo)=F_0$. Let $K\subset J$ be a compact interval with $\wo\in K$. For $q\in(0,1)$ put $\omega:=(1-q)\wo$; if $q$ is so close to $1$ that~\eqref{eq:bound} holds on $K$, let $\gamma_q\colon K\to\R^3$ be the curve given by \cref{thm:B} on $I=K$ with $\gamma_q(\wo)=p$ and $F_q(\wo)=F_0$. Then $\gamma_q\to\gamma_{\cl}$ and $F_q\to F_{\cl}$ uniformly on $K$ as $q\to1$. If $\kappa$ and $\tau$ are Lipschitz continuous on $K$, then $\abs{\gamma_q-\gamma_{\cl}}+\norm{F_q-F_{\cl}}=O(1-q)$ uniformly on $K$.
\end{proposition}

\begin{proof}
Let $A(t)$ be the skew-symmetric matrix $M_0$ of \cref{thm:B}(a) with $\kappa(\wo)$ and $\tau(\wo)$ replaced by $\kappa(t)$ and $\tau(t)$, so that $F_{\cl}'=AF_{\cl}$, and put $\Phi(x,t):=F_{\cl}(x)F_{\cl}(t)^{\top}\in\SO$. Let $m:=\max_K(\kappa+\abs{\tau})$, so that $\norm{A}\le m$ on $K$, let $\varpi$ be the modulus of continuity of $A$ on $K$, and put $\ell_K:=\max_{t\in K}\abs{t-\wo}$. Since $\partial_x\Phi(x,t)=A(x)\Phi(x,t)$, we have $\norm{\Phi(x,t)-I}\le m\abs{x-t}$ and
\[
  \bigl\lVert\Phi(t+h,t)-I-hA(t)\bigr\rVert=\Bigl\lVert\int_t^{t+h}\bigl(A(x)\Phi(x,t)-A(t)\bigr)\,dx\Bigr\rVert\le\abs{h}\bigl(\varpi(\abs{h})+m^2\abs{h}\bigr)
\]
for $t,t+h\in K$. On the other hand, let $(1-q)m\ell_K\le\frac12$, so that $m\abs{\delta}\le\frac12$ on $K$. Then the estimates of Step 2 of the proof of \cref{thm:B} hold on all of $K$ with $C_1:=m$, and together with~\eqref{eq:M}, $1-\cos\varphi=O(\delta^2)$ and $1/\cos\theta-1=O(\delta^2)$ they show that the step rotation $R(t)$ satisfies $\norm{(R(t)-I)/\delta(t)-A(t)}\le C\abs{\delta(t)}$ on $K^*:=K\setminus\{\wo\}$, where $C$ depends only on $m$. As $\sigma(t)=t+\delta(t)$, it follows that $\norm{R(t)-\Phi(\sigma(t),t)}\le\abs{\delta(t)}\,\epsilon(\abs{\delta(t)})$ with $\epsilon(h):=\varpi(h)+(m^2+C)h$. For orthogonal matrices $A_k$ and $B_k$ we have $\norm{A_{n-1}\cdots A_0-B_{n-1}\cdots B_0}\le\sum_k\norm{A_k-B_k}$. With $t_k:=\sigma^k(t)$, the product $\Pi(t)$ of Step 3 of the proof of \cref{thm:B} and $\Phi(\wo,t)=\lim_{n\to\infty}\Phi(t_n,t_{n-1})\cdots\Phi(t_1,t_0)$, \cref{lem:sigma}(c) therefore gives
\[
  \norm{\Pi(t)-\Phi(\wo,t)}\le\sum_{k\ge0}\abs{\delta(t_k)}\,\epsilon(\abs{\delta(t_k)})\le\epsilon\bigl((1-q)\ell_K\bigr)\,\ell_K\qquad(t\in K),
\]
which bounds $\norm{F_q(t)-F_{\cl}(t)}$, since $F_q(t)=\Pi(t)^{\top}F_0$ and $F_{\cl}(t)=\Phi(\wo,t)^{\top}F_0$. Finally, let $T_q$ and $T_{\cl}$ be the first rows of $F_q$ and $F_{\cl}$. By~\eqref{eq:JN},
\[
  \gamma_q(t)-\gamma_{\cl}(t)=-\sum_{k\ge0}\Bigl[\delta(t_k)\bigl(T_q(t_k)-T_{\cl}(t_k)\bigr)+\int_{t_k}^{t_{k+1}}\bigl(T_{\cl}(t_k)-T_{\cl}(x)\bigr)\,dx\Bigr],
\]
and $\abs{T_{\cl}'}=\kappa\le m$, so that $\abs{\gamma_q(t)-\gamma_{\cl}(t)}\le\ell_K\max_K\norm{F_q-F_{\cl}}+\frac m2(1-q)\ell_K^2$. If $\kappa$ and $\tau$ are Lipschitz continuous, then $\varpi(h)=O(h)$, and both errors are $O(1-q)$.
\end{proof}

\section{The classical limit and the fixed-point limit}\label{sec:limits}

\emph{Setting.} Let $J\subset\R$ be an open interval and $\gamma\in C^4(J;\R^3)$ a regular curve with $\gamma'\times\gamma''\neq0$ on $J$. We denote by $s_{\cl}=\abs{\gamma'}$ its classical speed, by $\kappa_{\cl}$ and $\tau_{\cl}$ its classical (arc-length) curvature and torsion, and by $F_{\cl}=(T_{\cl},N_{\cl},B_{\cl})$ its classical Frenet frame, so that $T_{\cl}'=s_{\cl}\kappa_{\cl}N_{\cl}$ and $B_{\cl}'=-s_{\cl}\tau_{\cl}N_{\cl}$ (in~\cite{doCarmo1976}, the torsion has the opposite sign). We fix $\wo\in J$ and, for each $q\in(0,1)$, take $\omega:=(1-q)\wo$, so that the fixed point is $\wo$ for every $q$. Let $K\subset J$ be a compact interval with $\wo\in K$ and $K^*:=K\setminus\{\wo\}$. For each $q$, $\gamma|_K$ is a Hahn curve; its Hahn quantities are denoted by $s_q$, $T_q$, $N_q$, $B_q$, $F_q$, $\kappa_q$, $\tau_q$, $\beta_q$, and its turning and torsion angles by $\theta$ and $\varphi$. For fixed $q$ the curve need not be Frenet-regular on all of $K$, but the formulas of \cref{def:frame} define these quantities at every $t\in K^*$ at which both pairs $T_q(t),T_q(\sigma(t))$ and $T_q(\sigma(t)),T_q(\sigma^2(t))$ are linearly independent, and \cref{thm:A} holds at such $t$, since its proof is pointwise. Throughout this section, $O(\abs{\delta}^k)$ denotes a quantity bounded by $C\abs{\delta(t)}^k$, where $C$ depends only on $\gamma$ and $K$ (and not on $q$ or $t$).

For $g\in C^4(K;\R^3)$ and $t_0,t_1,t_2\in K$ we use the divided differences $g[t_0,t_1]$ and $g[t_0,t_1,t_2]$, extended by continuity to coinciding arguments.

\begin{lemma}\label{lem:centroid}
Let $g\in C^4(K;\R^3)$, $t_0,t_1,t_2\in K$, $h:=\max_{i,j}\abs{t_i-t_j}$, $m:=(t_0+t_1)/2$ and $c:=(t_0+t_1+t_2)/3$. Then
\[
  \abs{g[t_0,t_1]-g'(m)}\le\tfrac{h^2}{24}\norm{g'''}_\infty,\qquad
  \bigl\lvert g[t_0,t_1,t_2]-\tfrac12g''(c)\bigr\rvert\le\tfrac{h^2}{4}\norm{g''''}_\infty,
\]
the supremum norms being taken over $K$. The first estimate only requires $g\in C^3(K;\R^3)$.
\end{lemma}

\begin{proof}
By the Hermite--Genocchi formula, $g[t_0,t_1]=\int_0^1g'(t_0+u(t_1-t_0))\,du$ and $g[t_0,t_1,t_2]=\int_Sg''(\lambda_0t_0+\lambda_1t_1+\lambda_2t_2)\,d\lambda$, where $S=\{\lambda\in[0,1]^3:\lambda_0+\lambda_1+\lambda_2=1\}$ carries the measure $d\lambda_1\,d\lambda_2$, of total mass $\frac12$. Expand $g'$ about $m$ and $g''$ about $c$ by Taylor's formula with integral remainder. The first-order terms vanish, because $\int_0^1(u-\frac12)\,du=0$ and $\int_S(\sum_i\lambda_it_i-c)\,d\lambda=0$ (each $\lambda_i$ has mean $\frac13$ on $S$). The remainders are bounded by $\frac12\norm{g'''}_\infty h^2\int_0^1(u-\frac12)^2\,du=\frac{h^2}{24}\norm{g'''}_\infty$ and by $\frac12\norm{g''''}_\infty h^2\cdot\frac12$, respectively.
\end{proof}

\begin{proposition}\label{prop:frames}
In the Setting, there are $C,\eta>0$, depending only on $\gamma$ and $K$, such that for all $q\in(0,1)$ and all $t\in K^*$ with $\abs{\delta(t)}\le\eta$ the following holds. The vectors $T_q(t)$ and $T_q(\sigma(t))$ are linearly independent, and, with $m(t):=\frac12(t+\sigma(t))$ and $c(t):=\frac13(t+\sigma(t)+\sigma^2(t))$,
\begin{gather*}
  \abs{s_q(t)-s_{\cl}(m(t))}\le C\delta(t)^2,\qquad \abs{T_q(t)-T_{\cl}(m(t))}\le C\delta(t)^2,\\
  \abs{B_q(t)-B_{\cl}(c(t))}\le C\delta(t)^2.
\end{gather*}
In particular, $\gamma$ is Frenet-regular on the $\sigma$-invariant interval $K_\eta:=\{t\in K:\abs{\delta(t)}\le\eta\}\ni\wo$.
\end{proposition}

\begin{proof}
We have $v_q(t)=\gamma[t,\sigma(t)]$, so \cref{lem:centroid} with $h=\abs{\delta(t)}$ gives $\abs{v_q(t)-\gamma'(m(t))}\le\frac1{24}\norm{\gamma'''}_\infty\delta^2$. The estimates for $s_q=\abs{v_q}$ and $T_q=v_q/\abs{v_q}$ follow, because $\abs{\gamma'}$ is bounded below on $K$. For the binormal put $t_i:=\sigma^i(t)$, $i=0,1,2$; then $t_2-t_0=(1+q)\delta(t)$, and since $\gamma[t_1,t_2]=\gamma[t_0,t_1]+(t_2-t_0)\gamma[t_0,t_1,t_2]$,
\begin{equation}\label{eq:crossdd}
  v_q(t)\times v_q(\sigma(t))=\gamma[t_0,t_1]\times\gamma[t_1,t_2]=(t_2-t_0)\,\gamma[t_0,t_1]\times\gamma[t_0,t_1,t_2].
\end{equation}
By \cref{lem:centroid} (with $h=(1+q)\abs{\delta}\le2\abs{\delta}$), $\gamma[t_0,t_1]=\gamma'(m)+O(\delta^2)$ and $\gamma[t_0,t_1,t_2]=\frac12\gamma''(c)+O(\delta^2)$. Moreover $\gamma'(m)=\gamma'(c)+(m-c)\gamma''(c)+O(\delta^2)$, and the term $(m-c)\gamma''(c)$ disappears in the cross product with $\gamma''(c)$. Therefore
\begin{equation}\label{eq:cross2}
  \gamma[t_0,t_1]\times\gamma[t_0,t_1,t_2]=\tfrac12\,\gamma'(c)\times\gamma''(c)+O(\delta^2).
\end{equation}
Since $\abs{\gamma'\times\gamma''}$ is bounded below on $K$, for $\abs{\delta}\le\eta$ with $\eta$ small the right-hand side of~\eqref{eq:cross2} does not vanish; by~\eqref{eq:crossdd}, $T_q(t)$ and $T_q(\sigma(t))$ are then linearly independent. Finally, as in the proof of \cref{thm:A}, $B_q(t)=\sgn(\delta)\,v_q(t)\times v_q(\sigma(t))/\abs{v_q(t)\times v_q(\sigma(t))}$, and $\sgn(t_2-t_0)=\sgn\delta(t)$; so~\eqref{eq:crossdd}--\eqref{eq:cross2} give $B_q(t)=B_{\cl}(c(t))+O(\delta^2)$. Moreover, on $K_\eta$ the field $T_q$ is continuous, and the estimate $T_q(t)=T_{\cl}(m(t))+O((t-\wo)^2)$ shows that it is differentiable at $\wo$, with $T_q'(\wo)=\frac{\qn{2}}{2}T_{\cl}'(\wo)\neq0$ (see the proof of \cref{cor:fixedpoint}); hence $\gamma$ is Frenet-regular on $K_\eta$.
\end{proof}

By \cref{lem:sigma}, for every $t$,
\begin{equation}\label{eq:centroidsteps}
  m(\sigma(t))-m(t)=\frac{\qn{2}}{2}\,\delta(t),\qquad c(\sigma(t))-c(t)=\frac{\qn{3}}{3}\,\delta(t).
\end{equation}

\begin{theorem}\label{thm:C}
In the Setting, there are $C,\eta>0$, depending only on $\gamma$ and $K$, such that for all $q\in(0,1)$ and all $t\in K^*$ with $\abs{\delta(t)}\le\eta$
\begin{gather*}
  \abs{s_q-s_{\cl}}\le C\abs{\delta},\qquad \abs{F_q-F_{\cl}}\le C\abs{\delta},\qquad
  \Bigl\lvert\kappa_q-\frac{\qn{2}}{2}\,s_{\cl}\kappa_{\cl}\Bigr\rvert\le C\abs{\delta},\\
  \Bigl\lvert\tau_q-\frac{\qn{3}}{3}\,s_{\cl}\tau_{\cl}\Bigr\rvert\le C\abs{\delta},\qquad
  \Bigl\lvert\beta_q+\frac{\qn{2}^2}{8}\,\delta\,s_{\cl}^2\kappa_{\cl}^2\Bigr\rvert\le C\delta^2,
\end{gather*}
all quantities being evaluated at $t$.
\end{theorem}

The constants do not depend on $q$; in particular, the estimates remain valid as $q\to0$, since $\sigma^k(t)$ always lies between $t$ and $\wo$.

\begin{proof}
The Hahn quantities at $t$ involve $\gamma$ at $t,\sigma(t),\sigma^2(t),\sigma^3(t)$, all of which lie between $t$ and $\wo$ and satisfy $\abs{\delta(\sigma^k(t))}\le\abs{\delta(t)}$; by \cref{prop:frames} they are well defined. The first two estimates follow from \cref{prop:frames}, since $m(t)-t$ and $c(t)-t$ are $O(\delta)$ and $N_q=B_q\times T_q$.

\emph{Curvature.} By~\eqref{eq:angles}, $\sin\theta=T_q(\sigma(t))\cdot N_q(t)=\bigl(T_q(\sigma(t))-T_q(t)\bigr)\cdot N_q(t)$. By \cref{prop:frames}, Taylor's formula and~\eqref{eq:centroidsteps},
\begin{align*}
  T_q(\sigma(t))-T_q(t)&=T_{\cl}(m(\sigma(t)))-T_{\cl}(m(t))+O(\delta^2)\\
  &=\frac{\qn{2}}{2}\,\delta\,(s_{\cl}\kappa_{\cl}N_{\cl})(m(t))+O(\delta^2),
\end{align*}
and $N_{\cl}(m(t))\cdot N_q(t)=1+O(\delta)$. Hence $\sin\theta=\frac{\qn{2}}{2}\delta\,s_{\cl}\kappa_{\cl}+O(\delta^2)$, with $s_{\cl}\kappa_{\cl}$ evaluated at $t$, and $\kappa_q=\sin\theta/\delta$ gives the claim.

\emph{Tangential component.} Shrinking $\eta$, we may assume $\cos\theta>0$. Then $\beta_q=(\cos\theta-1)/\delta=-\sin^2\theta/\bigl(\delta(1+\cos\theta)\bigr)=-\frac{\qn{2}^2}{8}\delta s_{\cl}^2\kappa_{\cl}^2+O(\delta^2)$.

\emph{Torsion.} From~\eqref{eq:rotation}, $B\times B(\sigma(t))=\sin\varphi\,T(\sigma(t))$ for every Frenet-regular curve; hence
\[
  \sin\varphi=\bigl(B_q(t)\times(B_q(\sigma(t))-B_q(t))\bigr)\cdot T_q(\sigma(t)).
\]
By \cref{prop:frames}, Taylor's formula and~\eqref{eq:centroidsteps}, $B_q(\sigma(t))-B_q(t)=-\frac{\qn{3}}{3}\delta\,(s_{\cl}\tau_{\cl}N_{\cl})(c(t))+O(\delta^2)$. Since $B_{\cl}\times N_{\cl}=-T_{\cl}$ and $T_q(\sigma(t))=T_{\cl}(c(t))+O(\delta)$, we obtain $\sin\varphi=\frac{\qn{3}}{3}\delta\,s_{\cl}\tau_{\cl}+O(\delta^2)$, and $\tau_q=\cos\theta\sin\varphi/\delta=\frac{\qn{3}}{3}s_{\cl}\tau_{\cl}+O(\delta)$ because $\cos\theta=1+O(\delta^2)$.
\end{proof}

\begin{corollary}\label{cor:classical}
In the Setting, let $q\to1$ with $\wo$ fixed (so that $\omega=(1-q)\wo\to0$). Then, uniformly on $K^*$,
\[
  s_q\to s_{\cl},\qquad F_q\to F_{\cl},\qquad \kappa_q\to s_{\cl}\kappa_{\cl},\qquad \tau_q\to s_{\cl}\tau_{\cl},\qquad \beta_q\to0,
\]
with errors $O(1-q)$. In particular $\kappa_q\to\kappa_{\cl}$ and $\tau_q\to\tau_{\cl}$ if $\gamma$ is parametrized by arc length.
\end{corollary}

\begin{proof}
On $K$, $\abs{\delta(t)}=(1-q)\abs{\wo-t}\le(1-q)\operatorname{diam}K$, which is at most $\eta$ for $q$ close to $1$. Moreover $\abs{\qn{2}/2-1}=(1-q)/2$ and $\abs{\qn{3}/3-1}=(1-q)(2+q)/3\le1-q$. Now apply \cref{thm:C}.
\end{proof}

\begin{remark}
The limit has to be taken with control of the fixed point: along paths $(q,\omega)\to(1,0)$ the limit of the fixed point $\wo=\omega/(1-q)$ depends on the path and need not exist: $\omega=(1-q)r$ gives $\wo\equiv r$ for every constant $r$, and $\omega=(1-q)\sin\bigl(1/(1-q)\bigr)$ gives a fixed point that oscillates as $q\to1$. This is why $\wo$ is kept fixed in \cref{cor:classical} and in \cref{prop:qto1}, which gives the corresponding statement for the curves of \cref{thm:B}, constructed from prescribed curvature and torsion.
\end{remark}

\begin{corollary}\label{cor:fixedpoint}
In the Setting, fix $q\in(0,1)$. As $t\to\wo$,
\begin{gather*}
  \kappa_q(t)\to\frac{\qn{2}}{2}\,s_{\cl}(\wo)\kappa_{\cl}(\wo),\qquad \tau_q(t)\to\frac{\qn{3}}{3}\,s_{\cl}(\wo)\tau_{\cl}(\wo),\\
  \beta_q(t)\to0,\qquad F_q(t)\to F_{\cl}(\wo).
\end{gather*}
Moreover $F_q$ is differentiable at $\wo$, so that the Hahn invariants at $\wo$ are defined, and
\begin{gather*}
  D^2\gamma(\wo)=\frac{\qn{2}}{2}\,\gamma''(\wo),\\
  \kappa_q(\wo)=\frac{\qn{2}}{2}\,s_{\cl}(\wo)\kappa_{\cl}(\wo),\qquad \tau_q(\wo)=\frac{\qn{3}}{3}\,s_{\cl}(\wo)\tau_{\cl}(\wo).
\end{gather*}
In particular, for $q<1$ we have $\kappa_q(\wo)\neq s_{\cl}(\wo)\kappa_{\cl}(\wo)$, and $\tau_q(\wo)\neq s_{\cl}(\wo)\tau_{\cl}(\wo)$ whenever $\tau_{\cl}(\wo)\neq0$.
\end{corollary}

\begin{proof}
The limits follow from \cref{thm:C}. By \cref{lem:sigma}, $m(t)=\wo+\frac{\qn{2}}{2}(t-\wo)$, $c(t)=\wo+\frac{\qn{3}}{3}(t-\wo)$ and $\delta(t)^2=(1-q)^2(t-\wo)^2$. By \cref{prop:frames} and its proof, for $t$ near $\wo$,
\begin{gather*}
  v_q(t)=\gamma'(m(t))+O(\delta^2),\qquad T_q(t)=T_{\cl}(m(t))+O(\delta^2),\\
  B_q(t)=B_{\cl}(c(t))+O(\delta^2).
\end{gather*}
An error term that vanishes at $\wo$ and is bounded by $C(t-\wo)^2$ is differentiable at $\wo$ with derivative $0$. Since $v_q(\wo)=\gamma'(\wo)$ and $T_q(\wo)=T_{\cl}(\wo)$, we obtain $D^2\gamma(\wo)=v_q'(\wo)=\frac{\qn{2}}{2}\gamma''(\wo)$ and $T_q'(\wo)=\frac{\qn{2}}{2}T_{\cl}'(\wo)$. Consequently $B_q(\wo)$, the unit vector along $T_q(\wo)\times T_q'(\wo)=\frac{\qn{2}}{2}T_{\cl}(\wo)\times T_{\cl}'(\wo)$, equals $B_{\cl}(\wo)$, and $N_q(\wo)=N_{\cl}(\wo)$. Now the estimate for $B_q$ gives $B_q'(\wo)=\frac{\qn{3}}{3}B_{\cl}'(\wo)$, and therefore $\kappa_q(\wo)=T_q'(\wo)\cdot N_{\cl}(\wo)=\frac{\qn{2}}{2}s_{\cl}(\wo)\kappa_{\cl}(\wo)$ and $\tau_q(\wo)=-B_q'(\wo)\cdot N_{\cl}(\wo)=\frac{\qn{3}}{3}s_{\cl}(\wo)\tau_{\cl}(\wo)$.
\end{proof}

The comparison in \cref{thm:C} is with the classical invariants of the curve $\gamma$ itself. If a Hahn curve is given through its Hahn unit tangent, as in \cref{thm:B}, it is natural to compare instead with the classical curve whose classical tangent is $T$.

\begin{proposition}\label{prop:tangentfield}
Let $T\in C^3(K;\R^3)$ with $\abs{T}\equiv1$ and $T'\neq0$ on $K$, let $\gamma$ be the unit-speed Hahn curve with $D\gamma=T$ (for instance $\gamma(t)=p+\int_{\wo}^{t}T(x)\,d_{q,\omega}x$), and let $\Gamma$ be a classical curve with $\Gamma'=T$, with curvature $\kappa_\Gamma=\abs{T'}$ and torsion $\tau_\Gamma=\det(T,T',T'')/\abs{T'}^2$. There are $C,\eta>0$, depending only on $T$ and $K$, such that for all $q\in(0,1)$ and all $t\in K^*$ with $\abs{\delta(t)}\le\eta$,
\[
  \abs{\kappa-\kappa_\Gamma}\le C\abs{\delta},\qquad \Bigl\lvert\tau-\frac{\qn{2}}{2}\,\tau_\Gamma\Bigr\rvert\le C\abs{\delta},\qquad \Bigl\lvert\beta+\frac{\delta}{2}\kappa_\Gamma^2\Bigr\rvert\le C\delta^2 .
\]
Moreover, $\gamma$ is Frenet-regular on $K_\eta:=\{t\in K:\abs{\delta(t)}\le\eta\}$, $\kappa(\wo)=\kappa_\Gamma(\wo)$ and $\tau(\wo)=\frac{\qn{2}}{2}\tau_\Gamma(\wo)$.
\end{proposition}

\begin{proof}
Let $(T,N_\Gamma,B_\Gamma)$ be the Frenet frame of $\Gamma$. By \cref{lem:centroid}, $T[t,\sigma(t)]=T'(m(t))+O(\delta^2)$, and $T(t)=T(m)-\frac\delta2T'(m)+O(\delta^2)$. Hence $T(t)\times T[t,\sigma(t)]=(T\times T')(m)+O(\delta^2)$, and since $T(t)\times T(\sigma(t))=\delta\,T(t)\times T[t,\sigma(t)]$,
\[
  B(t)=\sgn(\delta)\frac{T(t)\times T(\sigma(t))}{\abs{T(t)\times T(\sigma(t))}}=B_\Gamma(m(t))+O(\delta^2).
\]
In particular, $T(t)\times T(\sigma(t))\neq0$ if $\abs{\delta(t)}\le\eta$ with $\eta$ small, so $\gamma$ is Frenet-regular on $K_\eta$ (recall $T'(\wo)\neq0$). Consequently $N(t)=N_\Gamma(t)+O(\delta)$ and $\sin\theta=(T(\sigma(t))-T(t))\cdot N(t)=\delta\,T[t,\sigma(t)]\cdot N(t)=\delta\kappa_\Gamma(t)+O(\delta^2)$, which gives the estimates for $\kappa$ and $\beta$ as in the proof of \cref{thm:C}. The torsion is treated as there, with $c(t)$ replaced by $m(t)$ and $B_\Gamma'=-\tau_\Gamma N_\Gamma$, using~\eqref{eq:centroidsteps}. At $\wo$, $DT(\wo)=T'(\wo)$ gives $\kappa(\wo)=\abs{T'(\wo)}$ and $B(\wo)=B_\Gamma(\wo)$, so that $B(t)=B_\Gamma(m(t))+O((t-\wo)^2)$ gives $B'(\wo)=\frac{\qn{2}}{2}B_\Gamma'(\wo)$, as in the proof of \cref{cor:fixedpoint}.
\end{proof}

The factors in \cref{thm:C} disappear if the angles are divided by edge lengths instead of the parameter step. Normalizations by edge lengths are common for discrete curves. For polygons whose edges have length $1$, Sullivan~\cite{Sullivan2008} spreads the curvature at each vertex over the nearer halves of the two adjacent edges and compares the curvature densities $\theta$, $2\sin(\theta/2)$ and $2\tan(\theta/2)$, where $\theta$ is the turning angle. For arbitrary edge lengths, the same allocation gives as curvature density the turning angle divided by the mean length of the two adjacent edges; this quotient appears in \cref{prop:polygons}. Hoffmann~\cite{Hoffmann2009} defines the discrete curvature of arc-length parametrized polygons, whose edges have equal length, through $\tan(\theta/2)$. A different notion of discrete curvature and torsion, based on cross-ratios of consecutive vertices, was introduced by M\"uller and Vaxman~\cite{MullerVaxman2021}, who analyse its behaviour for sampled curves as the sampling density increases. For a Frenet-regular Hahn curve, with the edge lengths $\ell(t)=\abs{\gamma(\sigma(t))-\gamma(t)}$ of its orbit polygons (\cref{prop:reparam}), let
\begin{equation}\label{eq:edgeinv}
  \hat\kappa(t):=\frac{2\abs{\sin\theta(t)}}{\ell(t)+\ell(\sigma(t))},\qquad
  \hat\tau(t):=\frac{3\sgn(\delta(t))\sin\varphi(t)}{\ell(t)+\ell(\sigma(t))+\ell(\sigma^2(t))}\qquad(t\in\Is).
\end{equation}
The \emph{edge-normalized curvature} $\hat\kappa$ and \emph{torsion} $\hat\tau$ depend only on the orbit polygons, and by \cref{prop:reparam} they transform by composition with $h$ under admissible reparametrizations. By~\eqref{eq:invariants},
\begin{equation}\label{eq:edgeid}
  \kappa=\frac{\ell+\ell\circ\sigma}{2\abs{\delta}}\,\hat\kappa,\qquad
  \tau=\cos\theta\,\frac{\ell+\ell\circ\sigma+\ell\circ\sigma^2}{3\abs{\delta}}\,\hat\tau\qquad\text{on }\Is .
\end{equation}
Since $\ell(\sigma^k(t))=q^k\abs{\delta(t)}\,\abs{D\gamma(\sigma^k(t))}$, for a unit-speed curve these identities read $\kappa=\frac{\qn{2}}{2}\hat\kappa$ and $\tau=\frac{\qn{3}}{3}\cos\theta\,\hat\tau$. The same definitions apply in the Setting, wherever $\theta$ and $\varphi$ are defined.

\begin{corollary}\label{cor:edge}
In the Setting, there are $C,\eta>0$, depending only on $\gamma$ and $K$, such that $\abs{\hat\kappa-\kappa_{\cl}}\le C\abs{\delta}$ and $\abs{\hat\tau-\tau_{\cl}}\le C\abs{\delta}$ for all $q\in(0,1)$ and all $t\in K^*$ with $\abs{\delta(t)}\le\eta$. In particular, $\hat\kappa\to\kappa_{\cl}$ and $\hat\tau\to\tau_{\cl}$ in both limits of \cref{cor:classical,cor:fixedpoint}.
\end{corollary}

\begin{proof}
By \cref{thm:C}, $s_q(\sigma^k(t))=s_{\cl}(t)+O(\abs{\delta})$ for $k=0,1,2$, since $\abs{\sigma^k(t)-t}\le2\abs{\delta(t)}$. Hence the factors in~\eqref{eq:edgeid} are $\frac{\qn{2}}{2}s_{\cl}+O(\abs{\delta})$ and $\frac{\qn{3}}{3}s_{\cl}+O(\abs{\delta})$, where $\frac{\qn{2}}{2}\ge\frac12$, $\frac{\qn{3}}{3}\ge\frac13$ and $s_{\cl}$ is bounded below on $K$; moreover $\cos\theta=1+O(\delta^2)$. The claim now follows from~\eqref{eq:edgeid} and the expansions of $\kappa_q$ and $\tau_q$ in \cref{thm:C}.
\end{proof}

The same mechanism applies to arbitrary polygons inscribed in a smooth curve. Let $\gamma$ be as in the Setting, and let $(x_n)$ be a strictly monotone sequence in $J$ that converges to a point $x^*\in J$. For the polygon with vertices $P_n:=\gamma(x_n)$ let $\ell_n:=\abs{P_{n+1}-P_n}$ and $T_n:=(P_{n+1}-P_n)/\ell_n$, let $\theta_n\in[0,\pi]$ be the angle between $T_n$ and $T_{n+1}$, and, if $T_n\times T_{n+1}$ and $T_{n+1}\times T_{n+2}$ do not vanish, let $B_n:=T_n\times T_{n+1}/\abs{T_n\times T_{n+1}}$ and let $\varphi_n\in(-\pi,\pi]$ be the angle from $B_n$ to $B_{n+1}$ about $T_{n+1}$, i.e.\ $\cos\varphi_n=B_n\cdot B_{n+1}$ and $\sin\varphi_n=(B_n\times B_{n+1})\cdot T_{n+1}$. For the orbit polygon of a Frenet-regular Hahn curve with base point $t_0\in\Is$, these are the angles $\sgn(\delta(t_0))\,\theta(\sigma^n(t_0))$ and $\sgn(\delta(t_0))\,\varphi(\sigma^n(t_0))$ of \cref{def:frame} (for $t_0>\wo$, $T_n=-T(\sigma^n(t_0))$ and $B_n=-B(\sigma^n(t_0))$); thus~\eqref{eq:edgeinv} is the analogue of the quotients below with the angles replaced by their sines.

\begin{proposition}\label{prop:polygons}
In this situation $T_n\times T_{n+1}\neq0$ for all large $n$, and
\begin{gather*}
  \frac{2\theta_n}{\ell_n+\ell_{n+1}}=\kappa_{\cl}(x^*)+O(\abs{x_n-x^*}),\qquad
  \frac{3\varphi_n}{\ell_n+\ell_{n+1}+\ell_{n+2}}=\tau_{\cl}(x^*)+O(\abs{x_n-x^*}),
\end{gather*}
with constants that depend only on $\gamma$ near $x^*$. Consequently, if the ratios $\ell_{n+1}/\ell_n$ are bounded, then, with a constant that also depends on $\sup_n\ell_{n+1}/\ell_n$,
\[
  \frac{\varphi_n}{\theta_n}=\frac23\,\frac{\ell_n+\ell_{n+1}+\ell_{n+2}}{\ell_n+\ell_{n+1}}\,\frac{\tau_{\cl}(x^*)}{\kappa_{\cl}(x^*)}+O(\abs{x_n-x^*});
\]
in particular, $\lim_{n\to\infty}\varphi_n/\theta_n=\frac{2\qn{3}}{3\qn{2}}\,\tau_{\cl}(x^*)/\kappa_{\cl}(x^*)$ if $\ell_{n+1}/\ell_n\to q\in(0,1)$, and the limit is $\tau_{\cl}(x^*)/\kappa_{\cl}(x^*)$ if $\ell_{n+1}/\ell_n\to1$.
\end{proposition}

\begin{proof}
Put $r_n:=\abs{x_n-x^*}$, $h_n:=x_{n+1}-x_n$, $m_n:=(x_n+x_{n+1})/2$, $c_n:=(x_n+x_{n+1}+x_{n+2})/3$ and $H_n:=\abs{x_{n+3}-x_n}$, and let $\varsigma:=\sgn h_n$, which does not depend on $n$. The nodes $x_k$, $k\ge n$, lie between $x_n$ and $x^*$, so that $H_n\le r_n$, $\abs{m_{n+1}-m_n}=(\abs{h_n}+\abs{h_{n+1}})/2$ and $\abs{c_{n+1}-c_n}=H_n/3$. Since $P_{n+1}-P_n=h_n\gamma[x_n,x_{n+1}]$, \cref{lem:centroid} gives $T_n=\varsigma\,T_{\cl}(m_n)+O(h_n^2)$ and $\ell_n=s_{\cl}(x^*)\abs{h_n}\bigl(1+O(r_n)\bigr)$, so that
\begin{gather*}
  \ell_n+\ell_{n+1}=2s_{\cl}(x^*)\abs{m_{n+1}-m_n}\bigl(1+O(r_n)\bigr),\\
  \ell_n+\ell_{n+1}+\ell_{n+2}=3s_{\cl}(x^*)\abs{c_{n+1}-c_n}\bigl(1+O(r_n)\bigr).
\end{gather*}
The identity~\eqref{eq:crossdd} and the estimate~\eqref{eq:cross2} hold for arbitrary nodes, with the largest distance between the nodes in place of $\abs{\delta}$; since $h_kh_{k+1}(x_{k+2}-x_k)$ has the sign $\varsigma$, they give $B_k=\varsigma\,B_{\cl}(c_k)+O(H_n^2)$ for $k=n,n+1$, because the nodes $x_k,x_{k+1},x_{k+2}$ lie in an interval of length at most $H_n$. By the classical Frenet equations, the angle between $T_{\cl}(m_n)$ and $T_{\cl}(m_{n+1})$ equals $(s_{\cl}\kappa_{\cl})(x^*)\abs{m_{n+1}-m_n}\bigl(1+O(r_n)\bigr)$, and $B_{\cl}(c_n)\times B_{\cl}(c_{n+1})=(s_{\cl}\tau_{\cl})(x^*)\,(c_{n+1}-c_n)\,T_{\cl}(m_{n+1})+O(r_nH_n)$. As $\abs{h_n},\abs{h_{n+1}}\le2\abs{m_{n+1}-m_n}$ and $\varsigma\,(c_{n+1}-c_n)=\abs{c_{n+1}-c_n}$, we obtain $\theta_n=(s_{\cl}\kappa_{\cl})(x^*)\abs{m_{n+1}-m_n}\bigl(1+O(r_n)\bigr)$, which is positive for large $n$, and $\sin\varphi_n=(B_n\times B_{n+1})\cdot T_{n+1}=(s_{\cl}\tau_{\cl})(x^*)\abs{c_{n+1}-c_n}+O(r_nH_n)$, while $\varphi_n\to0$. Dividing by the sums of edge lengths gives the two estimates. If the ratios $\ell_{n+1}/\ell_n$ are bounded, then so is $\Lambda_n:=(\ell_n+\ell_{n+1}+\ell_{n+2})/(\ell_n+\ell_{n+1})$, and the quotient of the two estimates gives the formula for $\varphi_n/\theta_n$, because $\kappa_{\cl}(x^*)>0$. Finally, $\Lambda_n$ tends to $(1+q+q^2)/(1+q)$ if $\ell_{n+1}/\ell_n\to q$, and to $\frac32$ if $\ell_{n+1}/\ell_n\to1$.
\end{proof}

\begin{remark}\label{rem:qnumbers}
The vectors $T$ and $B$ are determined by $k$ consecutive points of an orbit, namely by $k=2$ and $k=3$ points for the chord field of a smooth curve and by $k=1$ and $k=2$ values of a smooth tangent field. They approximate the corresponding classical vectors at the centroid of the parameters of these points with an error $O(\delta^2)$ (\cref{lem:centroid,prop:frames,prop:tangentfield}), and the centroids of $k$ consecutive points at $t$ and at $\sigma(t)$ are $\frac{\qn{k}}{k}\delta(t)$ apart (\cref{lem:sigma}). The Hahn curvature and torsion compare two such vectors one step apart, but divide by $\delta(t)$; this gives the factors in \cref{thm:C,prop:tangentfield}. For $q=1$, i.e.\ for uniform parameter steps, all factors $\qn{k}/k$ equal $1$, which is why these effects are invisible for polygons with equal steps. Dividing by the mean length of two or three consecutive edges instead removes the factors of \cref{thm:C} (\cref{cor:edge,prop:polygons}). The orbit polygons of a Hahn curve are inscribed in the curve itself: if $\gamma=p+\int_{\wo}^{t}T(x)\,d_{q,\omega}x$ is given by a smooth tangent field $T$, then $\gamma$ is again smooth, as termwise differentiation of the series~\eqref{eq:JN} shows, and $\hat\kappa$ and $\hat\tau$ approximate the classical invariants of $\gamma$, not those of the curve $\Gamma$ with $\Gamma'=T$ used in \cref{prop:tangentfield}. The identity $D^2\gamma(\wo)=\frac{\qn{2}}{2}\gamma''(\wo)$ is the analogue for curves of $D_{q,0}^2\,t^2=\qn{2}$.

We nevertheless use $\kappa$ and $\tau$ as the primary invariants, for three reasons. First, they are the entries of the matrix in the Hahn Frenet equations $DF=MF$, on which the fundamental theorem rests, and at $\wo$ no division by $\delta$ occurs at all, since $\kappa(\wo)$ and $\tau(\wo)$ are computed from the classical derivatives $T'(\wo)$ and $B'(\wo)$ of the Hahn frame (\cref{prop:fixedpoint}). Second, the quotient of the two factors is not affected by any common normalization of $\kappa$ and $\tau$, and it is already visible in the angles of the orbit polygons. Indeed, $\delta$ cancels in $\tau/\kappa=\cos\theta\sin\varphi/\sin\theta$, and whenever $B$ is differentiable at $\wo$, $\tau(\wo)/\kappa(\wo)=\lim_{t\to\wo}\varphi(t)/\theta(t)$ by \cref{prop:fixedpoint}, a limit of ratios of angles of the orbit polygons, and the same factor $2\qn{3}/(3\qn{2})$ appears in $\lim_{n\to\infty}\varphi_n/\theta_n$ for every polygon as in \cref{prop:polygons} whose edge lengths decay with ratio $q$. By \cref{cor:fixedpoint}, for every regular $C^4$ reparametrization $\tilde\gamma=\gamma\circ\nu$ (a different Hahn curve unless $\nu$ is a linear change of parameter about $\wo$; see \cref{rem:coordfree}), with Hahn invariants $\tilde\kappa_q$ and $\tilde\tau_q$, we have
\[
  \frac{\tilde\tau_q(\wo)}{\tilde\kappa_q(\wo)}=\frac{2\qn{3}}{3\qn{2}}\,\frac{\tau_{\cl}}{\kappa_{\cl}}\bigl(\nu(\wo)\bigr),
\]
so that neither a change of parameter nor a common normalization of $\kappa$ and $\tau$ removes the factor $2\qn{3}/(3\qn{2})\neq1$. Third, the Lancret ratio of generalized Hahn helices is deformed even for the edge-normalized invariants (\cref{cor:lancretratio}).
\end{remark}

\section{Generalized Hahn helices: an intrinsic Lancret theorem}\label{sec:helices}

\begin{definition}\label{def:helix}
Let $\vartheta\in(0,\pi/2)$. A Frenet-regular Hahn curve $\gamma$ is a \emph{generalized Hahn helix} with \emph{axis} $u\in\R^3$, $\abs{u}=1$, and \emph{angle} $\vartheta$ if $T(t)\cdot u=\cos\vartheta$ for all $t\in I$.
\end{definition}

For polygons, analogues of parts (i)--(iii) of the following lemma, in a different indexing convention and with unsigned turning angles, are contained in the conserved-vector characterization of discrete general helices by Yoon, Lee and Lee~\cite[Proposition~3.6 and Corollaries~4.4 and~4.8]{YoonLeeLee2026}. We include the short proof, since we need the signed angles of \cref{def:frame} and the values at $\wo$.

\begin{lemma}\label{lem:helix}
Let $\gamma$ be a generalized Hahn helix with axis $u$ and angle $\vartheta$, and put $a:=\cos\vartheta$, $b:=N\cdot u$, $c:=B\cdot u$. Then for $t\in\Is$:
\begin{enumerate}[label=\textup{(\roman*)}]
\item $b(t)=a\tan\frac{\theta(t)}{2}=-a\,\beta(t)/\kappa(t)$;
\item $c(t)^2=1-a^2\sec^2\frac{\theta(t)}{2}$; in particular $\abs{\theta(t)}\le2\vartheta$;
\item $b(\sigma(t))=-\cos\varphi(t)\,b(t)+\sin\varphi(t)\,c(t)$ and $c(\sigma(t))=\sin\varphi(t)\,b(t)+\cos\varphi(t)\,c(t)$.
\end{enumerate}
Moreover $b(\wo)=0$ and $c(\wo)^2=\sin^2\vartheta$.
\end{lemma}

\begin{proof}
(i) By~\eqref{eq:rotation}, $a=T(\sigma(t))\cdot u=a\cos\theta+b\sin\theta$, so $b=a(1-\cos\theta)/\sin\theta=a\tan(\theta/2)$, and $\tan(\theta/2)=(1-\cos\theta)/\sin\theta=-\delta\beta/(\delta\kappa)$ by~\eqref{eq:invariants}. (ii) follows from $a^2+b^2+c^2=1$, and $c^2\ge0$ gives $\cos^2(\theta/2)\ge\cos^2\vartheta$, i.e.\ $\abs{\theta}/2\le\vartheta$. (iii) Take the scalar product of the second and third lines of~\eqref{eq:rotation} with $u$, and use the identities $-a\sin\theta+b\cos\theta=-b$ and $a\sin\theta-b\cos\theta=b$, which are equivalent to $b(1+\cos\theta)=a\sin\theta$, i.e.\ to (i). At $\wo$, $DT(\wo)=T'(\wo)$ is orthogonal to $u$ (since $T\cdot u$ is constant) and equals $\kappa(\wo)N(\wo)$ with $\kappa(\wo)>0$ by \cref{prop:fixedpoint}(a); hence $b(\wo)=0$.
\end{proof}

\begin{theorem}[intrinsic Lancret theorem]\label{thm:D}
Let $\gamma$ be a Frenet-regular Hahn curve, $\vartheta\in(0,\pi/2)$ and $a:=\cos\vartheta$. The following are equivalent:
\begin{enumerate}[label=\textup{(\roman*)}]
\item $\gamma$ is a generalized Hahn helix with angle $\vartheta$ and some axis $u$ such that $B(t)\cdot u\neq0$ for all $t\in I$;
\item $\abs{\theta(t)}<2\vartheta$ for all $t\in\Is$, and there is $\varepsilon\in\{-1,1\}$ such that for all $t\in\Is$
\begin{equation}\label{eq:H}
  a\Bigl[\tan\frac{\theta(\sigma(t))}{2}+\cos\varphi(t)\tan\frac{\theta(t)}{2}\Bigr]=\varepsilon\sin\varphi(t)\sqrt{1-a^2\sec^2\frac{\theta(t)}{2}} .
\end{equation}
\end{enumerate}
In this case $\varepsilon=\sgn(B\cdot u)$ and $u=a\,T(\wo)+\varepsilon\sin\vartheta\,B(\wo)$.
\end{theorem}

\begin{proof}
(i)$\Rightarrow$(ii) By \cref{lem:helix}(ii), $c^2=1-a^2\sec^2(\theta/2)$, and $c=B\cdot u\neq0$, so $\abs{\theta}<2\vartheta$. Since $c$ is continuous on $I$ and does not vanish, $c=\varepsilon\sqrt{1-a^2\sec^2(\theta/2)}$ with the constant sign $\varepsilon=\sgn c$. Inserting \cref{lem:helix}(i) into the first relation of \cref{lem:helix}(iii) gives~\eqref{eq:H}. The formula for $u$ follows from $u=a\,T+b\,N+c\,B$ at $\wo$, where $b(\wo)=0$ and $c(\wo)=\varepsilon\sin\vartheta$.

(ii)$\Rightarrow$(i) On $\Is$ put $b:=a\tan(\theta/2)$ and $c:=\varepsilon\sqrt{1-a^2\sec^2(\theta/2)}$ (note $c\neq0$ because $\abs{\theta}<2\vartheta$), set $b(\wo):=0$, $c(\wo):=\varepsilon\sin\vartheta$, and $W:=aT+bN+cB$. Since $\cos\theta$ and $\sin\theta$ are continuous on $\Is$ by~\eqref{eq:angles} and $\sin\theta\neq0$, the function $\theta$ is continuous on $\Is$; moreover $\theta(t)\to0$ as $t\to\wo$, because $\cos\theta(t)\to1$. Hence $b$, $c$ and $W$ are continuous on $I$, and $\abs{W}=1$. For $t\in\Is$, \eqref{eq:rotation} and the identities in the proof of \cref{lem:helix} give
\begin{gather*}
  T(\sigma(t))\cdot W(t)=a,\quad
  N(\sigma(t))\cdot W(t)=-b\cos\varphi+c\sin\varphi=:\hat b(t),\quad
  B(\sigma(t))\cdot W(t)=b\sin\varphi+c\cos\varphi=:\hat c(t).
\end{gather*}
By~\eqref{eq:H}, $\hat b(t)=b(\sigma(t))$. Since $a^2+\hat b^2+\hat c^2=1=a^2+b(\sigma(t))^2+c(\sigma(t))^2$, we have $\hat c(t)^2=c(\sigma(t))^2>0$. The function $\hat c$ is continuous on each component of $\Is$, does not vanish, and tends to $\varepsilon\sin\vartheta$ as $t\to\wo$ (because $b\to0$, $c\to\varepsilon\sin\vartheta$ and $\varphi\to0$). Hence $\sgn\hat c=\varepsilon=\sgn c(\sigma(t))$, so $\hat c(t)=c(\sigma(t))$ and therefore $W(t)=aT(\sigma(t))+b(\sigma(t))N(\sigma(t))+c(\sigma(t))B(\sigma(t))=W(\sigma(t))$. By \cref{lem:JN}(c), $W$ is constant; with $u:=W(\wo)$ we get $T\cdot u\equiv a$ and $B\cdot u=c\neq0$.
\end{proof}

Relation~\eqref{eq:H} involves only $\theta(t)$, $\theta(\sigma(t))$ and $\varphi(t)$, i.e.\ two consecutive steps of each orbit polygon, and no reference to the axis. Since $\tan(\theta/2)=-\beta/\kappa$ by~\eqref{eq:invariants}, the turning angles enter only through the ratio $\beta/\kappa$; under condition~\eqref{eq:bound}, \cref{cor:relations}(d) moreover expresses $\beta$, $\sin\varphi$ and $\cos\varphi$ through $\kappa$, $\tau$ and $\delta$, so that~\eqref{eq:H} becomes a relation between the Hahn curvature and torsion at $t$ and $\sigma(t)$.

The condition $B\cdot u\neq0$ in \cref{thm:D}(i) excludes the degenerate case in which $\abs{\theta(t)}=2\vartheta$ at some point (\cref{lem:helix}(ii)). It holds automatically near $\wo$, where $\theta\to0$. If $B\cdot u$ has zeros, \cref{lem:helix} still applies, and~\eqref{eq:H} holds with $\varepsilon$ replaced by $\sgn(B(t)\cdot u)$, which may change sign at the zeros of $B\cdot u$; we do not pursue this case.

\begin{corollary}\label{cor:allhelices}
Let $\vartheta\in(0,\pi/2)$, $a=\cos\vartheta$, $\varepsilon\in\{-1,1\}$, and let $\theta\colon I\to\R$ be continuous with $\theta(\wo)=0$, $0<\sgn(\delta(t))\,\theta(t)<2\vartheta$ for $t\in\Is$, and $\lim_{t\to\wo}\theta(t)/\delta(t)=\kappa_0>0$. Define $b$ and $c$ on $I$ as in the proof of \cref{thm:D}, and $\varphi\colon I\to(-\pi,\pi)$ by $\varphi(\wo):=0$ and, for $t\in\Is$,
\begin{equation}\label{eq:phihelix}
  \sin\varphi(t)=\frac{c(t)\,b(\sigma(t))+b(t)\,c(\sigma(t))}{\sin^2\vartheta},\qquad
  \cos\varphi(t)=\frac{c(t)\,c(\sigma(t))-b(t)\,b(\sigma(t))}{\sin^2\vartheta}.
\end{equation}
Then $\varphi$ is continuous, $\varphi(t)/\delta(t)\to\varepsilon\,\frac{\qn{2}}{2}\,\kappa_0\cot\vartheta$ as $t\to\wo$, and the curve given by \cref{thm:Bangle} with angles $(\theta,\varphi)$ is a generalized Hahn helix with angle $\vartheta$. Conversely, the torsion angle of every generalized Hahn helix as in \cref{thm:D}\textup{(i)} is given by~\eqref{eq:phihelix}. Thus, up to proper rigid motions, unit-speed generalized Hahn helices with angle $\vartheta$ and sign $\varepsilon$ correspond bijectively to their turning-angle functions; a helix with prescribed Hahn speed is then obtained as in \cref{rem:speed}(i).
\end{corollary}

\begin{proof}
Since $b^2+c^2=\sin^2\vartheta$ on $I$, \eqref{eq:phihelix} defines $\varphi(t)$ as the angle of the plane rotation that maps $(b(t),c(t))$ onto $(-b(\sigma(t)),c(\sigma(t)))$; a direct computation gives
\begin{equation}\label{eq:rotbc}
  b(\sigma(t))=-b(t)\cos\varphi(t)+c(t)\sin\varphi(t),\qquad c(\sigma(t))=b(t)\sin\varphi(t)+c(t)\cos\varphi(t).
\end{equation}
Because $c(t)$ and $c(\sigma(t))$ have the same sign $\varepsilon$, the two vectors are never opposite, so $\cos\varphi\neq-1$ and $\varphi$ is continuous on $\Is$. As $t\to\wo$, $b(t),b(\sigma(t))\to0$ and $c(t),c(\sigma(t))\to\varepsilon\sin\vartheta$, so $\varphi\to0$. Furthermore $b(t)/\delta(t)\to a\kappa_0/2$ and $b(\sigma(t))/\delta(t)=q\,b(\sigma(t))/\delta(\sigma(t))\to qa\kappa_0/2$, so that $\sin\varphi/\delta\to\varepsilon\sin\vartheta(1+q)a\kappa_0/(2\sin^2\vartheta)=\varepsilon\frac{\qn{2}}{2}\kappa_0\cot\vartheta$; the same limit holds for $\varphi/\delta$. Thus $(\theta,\varphi)$ satisfy the hypotheses of \cref{thm:Bangle}. For the resulting curve the first relation in~\eqref{eq:rotbc} is~\eqref{eq:H}, and $\abs{\theta}<2\vartheta$; so \cref{thm:D} shows that it is a generalized Hahn helix with angle $\vartheta$. Conversely, for a generalized Hahn helix as in \cref{thm:D}\textup{(i)}, \cref{lem:helix}(iii) is~\eqref{eq:rotbc}, which determines the rotation angle $\varphi(t)$ uniquely; moreover its turning angle satisfies the hypotheses of the corollary, with $\kappa_0=\kappa(\wo)$, by \cref{thm:A,thm:D} and \cref{prop:fixedpoint}(a). The last statement follows from the uniqueness part of \cref{thm:Bangle}.
\end{proof}

\begin{corollary}\label{cor:lancretratio}
Let $\gamma$ be as in \cref{thm:D}\textup{(i)}. Then $B$ is differentiable at $\wo$, and, with $\varepsilon=\sgn(B\cdot u)$,
\[
  \frac{\tau(\wo)}{\kappa(\wo)}=\varepsilon\,\frac{\qn{2}}{2}\cot\vartheta,\qquad
  \lim_{t\to\wo}\frac{\hat\tau(t)}{\hat\kappa(t)}=\varepsilon\,\frac{3\qn{2}^2}{4\qn{3}}\cot\vartheta .
\]
In particular, for $q<1$ the classical Lancret relation $\tau/\kappa=\varepsilon\cot\vartheta$ fails at the fixed point, also for the edge-normalized invariants~\eqref{eq:edgeinv}, since $4\qn{3}-3\qn{2}^2=(1-q)^2$.
\end{corollary}

\begin{proof}
Put $b:=N\cdot u$ and $c:=B\cdot u$. Since $u=aT+bN+cB$, $T\times u=bB-cN$ and $b^2+c^2=\sin^2\vartheta$, we have
\[
  B=\frac{c\,(u-aT)+b\,T\times u}{\sin^2\vartheta}\qquad\text{on }I.
\]
By \cref{lem:helix} and the proof of \cref{thm:D}, $b=a\tan(\theta/2)$ and $c=\varepsilon(1-a^2\sec^2(\theta/2))^{1/2}$ on $\Is$, while $b(\wo)=0$ and $c(\wo)=\varepsilon\sin\vartheta$. As $\theta(t)/\delta(t)\to\kappa(\wo)$ by \cref{prop:fixedpoint}(a) and $\delta(t)=-(1-q)(t-\wo)$, the function $b$ is differentiable at $\wo$ with $b'(\wo)=-(1-q)a\kappa(\wo)/2$; moreover $c(t)-c(\wo)=O(\theta(t)^2)$, so that $c'(\wo)=0$. Since $T$ is differentiable at $\wo$ with $T'(\wo)=\kappa(\wo)N(\wo)$, and $(T\times u)(\wo)=-\varepsilon\sin\vartheta\,N(\wo)$, it follows that $B$ is differentiable at $\wo$ with
\[
  B'(\wo)=-\frac{\varepsilon\sin\vartheta\,\bigl(a\kappa(\wo)+b'(\wo)\bigr)}{\sin^2\vartheta}\,N(\wo)=-\varepsilon\,\frac{(1+q)\,a\,\kappa(\wo)}{2\sin\vartheta}\,N(\wo),
\]
and $\tau(\wo)=-B'(\wo)\cdot N(\wo)$ gives the first formula. By~\eqref{eq:edgeid} and $\ell=\abs{\delta}s$,
\begin{equation*}
  \frac{\hat\tau}{\hat\kappa}=\frac32\,\frac{\tau}{\kappa}\,\frac{s+q\,s\circ\sigma}{\cos\theta\,\bigl(s+q\,s\circ\sigma+q^2\,s\circ\sigma^2\bigr)}\qquad
  \text{on $\Is$ wherever $\cos\theta\neq0$, in particular near $\wo$;}
\end{equation*}
since $s$ is continuous at $\wo$ (\cref{lem:continuityDf}), $\cos\theta\to1$ and $\tau/\kappa\to\tau(\wo)/\kappa(\wo)$ (\cref{prop:fixedpoint}), the right-hand side tends to $\frac{3\qn{2}}{2\qn{3}}\,\tau(\wo)/\kappa(\wo)$.
\end{proof}

\begin{remark}\label{rem:helixremarks}
(a) If a generalized Hahn helix $\gamma$ is unit-speed and its Hahn tangent field $T$ is smooth, then the classical curve $\Gamma$ with $\Gamma'=T$ is a classical general helix, so $\tau_\Gamma/\kappa_\Gamma=\varepsilon\cot\vartheta$ by Lancret's theorem, and \cref{prop:tangentfield} gives $\kappa(\wo)=\kappa_\Gamma(\wo)$ and $\tau(\wo)=\frac{\qn{2}}{2}\tau_\Gamma(\wo)$, in accordance with \cref{cor:lancretratio}. The orbit polygons, however, are inscribed in $\gamma$ itself, which is smooth with $\gamma'(\wo)=T(\wo)$ and $\gamma''(\wo)=\frac{2}{\qn{2}}T'(\wo)$ (termwise differentiation of~\eqref{eq:JN}), and by \cref{cor:edge,cor:lancretratio} the classical curvature and torsion of $\gamma$ satisfy $\tau_{\cl}(\wo)/\kappa_{\cl}(\wo)=\varepsilon\frac{3\qn{2}^2}{4\qn{3}}\cot\vartheta$. Thus the edges of the orbit polygons make the constant angle $\vartheta$ with $u$, but for $q<1$ the curve $\gamma$ is not a classical general helix with angle $\vartheta$.
(b) For smooth data, dividing~\eqref{eq:H} by $\delta$ and letting $q\to1$ with $\wo$ fixed yields the classical Lancret relation $\tau/\kappa=\varepsilon\cot\vartheta$; for uniform steps, the second-order consistency of the polygonal Lancret quotient is proved in~\cite[Proposition~5.5]{YoonLeeLee2026}.
(c) Away from $\wo$ the ratio $\tau/\kappa$ of a generalized Hahn helix is in general not constant; see \cref{ex:hahnhelix}.
\end{remark}

By \cref{rem:helixremarks}(c), the ratio $\tau/\kappa$ of a generalized Hahn helix is in general not constant. Conversely, a constant ratio does not characterize generalized Hahn helices. To describe the helices with constant ratio, fix $\vartheta\in(0,\pi/2)$ and $\varepsilon\in\{-1,1\}$, and put $a:=\cos\vartheta$ and $\rho:=\varepsilon\frac{\qn{2}}{2}\cot\vartheta$. For $\abs{x}<x_\vartheta:=\min\{2\vartheta,\arctan(1/\abs{\rho})\}$ define $f_\vartheta(x)\in(-\pi,\pi)$ by
\begin{equation}\label{eq:Fc}
  a\tan\frac{f_\vartheta(x)}{2}=\frac{\qn{2}}{2}\cot\vartheta\,\tan x\,\sqrt{1-a^2\sec^2\tfrac{x}{2}}-a\sqrt{1-\rho^2\tan^2x}\,\tan\frac{x}{2};
\end{equation}
here $\frac{\qn{2}}{2}\cot\vartheta=\varepsilon\rho$, and $f_\vartheta$ does not depend on $\varepsilon$, since $\rho$ enters~\eqref{eq:Fc} and $x_\vartheta$ only through $\rho^2$ and $\abs{\rho}$. For a Frenet-regular curve, $\abs{\theta}<\pi$ on $\Is$ by \cref{thm:A}; hence, if $\abs{\theta(t)}<x_\vartheta$, $\sin\varphi(t)=\rho\tan\theta(t)$ and $\cos\varphi(t)>0$, then~\eqref{eq:H} with the sign $\varepsilon$ is equivalent to $\theta(\sigma(t))=f_\vartheta(\theta(t))$. The function $f_\vartheta$ is real-analytic, with $f_\vartheta(0)=0$ and $f_\vartheta'(0)=2\varepsilon\rho\tan\vartheta-1=q$. Fix an open interval $U_\vartheta\ni0$ contained in $(-x_\vartheta,x_\vartheta)$ on which $f_\vartheta'>0$ and $0<f_\vartheta(x)/x<1$ for $x\neq0$; such intervals exist because $0<f_\vartheta'(0)<1$, and then $f_\vartheta(U_\vartheta)\subset U_\vartheta$. ($U_\vartheta$ cannot always be taken to be $(-x_\vartheta,x_\vartheta)$, since $f_\vartheta$ may have further fixed points there.) By the proof of \cref{prop:coordfree}, applied to $\Sigma:=f_\vartheta|_{U_\vartheta}$, the Koenigs function $\mathcal K_\vartheta:=\lim_{n\to\infty}q^{-n}f_\vartheta^n$ is an increasing $C^1$-diffeomorphism of $U_\vartheta$ onto an interval $\mathcal K_\vartheta(U_\vartheta)\ni0$, with $\mathcal K_\vartheta\circ f_\vartheta=q\,\mathcal K_\vartheta$ and $\mathcal K_\vartheta'(0)=1$~\cite{Koenigs1884,KuczmaChoczewskiGer1990}.

\begin{proposition}\label{prop:constratio}
Let $\vartheta$, $\varepsilon$, $a$, $\rho$, $f_\vartheta$, $\mathcal K_\vartheta$ and $U_\vartheta$ be as above.
\begin{enumerate}[label=\textup{(\alph*)}]
\item Let $\gamma$ be a generalized Hahn helix with angle $\vartheta$ and axis $u$ such that $\sgn(B(\wo)\cdot u)=\varepsilon$ and $\tau/\kappa$ is constant on $\Is$. Then $\tau/\kappa\equiv\rho$, and there is a $\sigma$-invariant neighbourhood $I_0$ of $\wo$ in $I$ with $\kappa(\wo)\,\delta(I_0)\subset\mathcal K_\vartheta(U_\vartheta)$ such that $\theta=\mathcal K_\vartheta^{-1}\circ(\kappa(\wo)\,\delta)$ and $\varphi=\arcsin(\rho\tan\theta)$ on $I_0\setminus\{\wo\}$. In particular, if $\gamma$ is unit-speed, then $\gamma|_{I_0}$ is, up to a proper rigid motion, the helix of \textup{(b)} with $\kappa_0=\kappa(\wo)$ and $I$ replaced by $I_0$.
\item Conversely, let $\kappa_0>0$, and let $I\ni\wo$ be an interval with $\kappa_0\,\delta(I)\subset\mathcal K_\vartheta(U_\vartheta)$. Then $\theta:=\mathcal K_\vartheta^{-1}\circ(\kappa_0\delta)$ and $\varphi:=\arcsin(\rho\tan\theta)$ satisfy the hypotheses of \cref{thm:Bangle}, and the corresponding unit-speed Hahn curve $\gamma$ is a generalized Hahn helix with angle $\vartheta$ and an axis $u$ such that $\sgn(B\cdot u)=\varepsilon$; moreover $\kappa(\wo)=\kappa_0$ and $\tau/\kappa\equiv\rho$. On $h^{-1}(I)$, where $h(t):=\wo+(\tilde\kappa_0/\kappa_0)(t-\wo)$, the helix obtained from another value $\tilde\kappa_0>0$ is, up to a proper rigid motion, $(\kappa_0/\tilde\kappa_0)\,\gamma\circ h$.
\item A Frenet-regular Hahn curve with constant Hahn curvature and constant Hahn torsion on $\Is$ is not a generalized Hahn helix, with any angle in $(0,\pi/2)$, on any $\sigma$-invariant neighbourhood of $\wo$.
\end{enumerate}
\end{proposition}

\begin{proof}
(a) By \cref{lem:helix}, $(B(\wo)\cdot u)^2=\sin^2\vartheta>0$. Since $B$ is continuous, $\theta(t)\to0$ (\cref{prop:fixedpoint}(a)) and $\cos\varphi(t)=B(t)\cdot B(\sigma(t))\to1$ as $t\to\wo$, there is a $\sigma$-invariant neighbourhood $I_0=I\cap(\wo-r,\wo+r)$ of $\wo$ on which $\sgn(B\cdot u)=\varepsilon$, $\abs{\varphi}<\pi/2$ and $\theta\in U_\vartheta$. By \cref{thm:D}, applied to $\gamma|_{I_0}$, \eqref{eq:H} holds on $I_0^*:=I_0\setminus\{\wo\}$. Let $\rho_0$ be the constant value of $\tau/\kappa$; by~\eqref{eq:invariants}, $\sin\varphi=\rho_0\tan\theta$, and $\cos\varphi=(1-\rho_0^2\tan^2\theta)^{1/2}$ on $I_0^*$. By \cref{cor:lancretratio}, applied to $\gamma|_{I_0}$, $B$ is differentiable at $\wo$ and $\tau(\wo)/\kappa(\wo)=\rho$; since $\tau(t)\to\tau(\wo)$ and $\kappa(t)\to\kappa(\wo)$ as $t\to\wo$ by \cref{prop:fixedpoint}, we get $\rho_0=\rho$. Hence $\theta(\sigma(t))=f_\vartheta(\theta(t))$ for $t\in I_0^*$, and $G:=(\mathcal K_\vartheta\circ\theta)/\delta$ satisfies $G\circ\sigma=G$ on $I_0^*$ and $G(t)\to\kappa(\wo)$ as $t\to\wo$, because $\mathcal K_\vartheta(x)/x\to1$ and $\theta/\delta\to\kappa(\wo)$. By \cref{lem:JN}(c), applied on $I_0$ to $G$ extended by $G(\wo):=\kappa(\wo)$, we get $G\equiv\kappa(\wo)$, i.e.\ $\theta=\mathcal K_\vartheta^{-1}\circ(\kappa(\wo)\,\delta)$ on $I_0^*$; in particular $\kappa(\wo)\,\delta(I_0)\subset\mathcal K_\vartheta(U_\vartheta)$, since $\theta(I_0^*)\subset U_\vartheta$ and $\mathcal K_\vartheta(0)=0$. The formula for $\varphi$ follows from $\sin\varphi=\rho\tan\theta$ and $\abs{\varphi}<\pi/2$, and the last claim from the uniqueness part of \cref{thm:Bangle}.

(b) By construction, $\theta$ is continuous on $I$, $\theta(\wo)=0$, $\sgn\theta=\sgn\delta$ because $\mathcal K_\vartheta$ is increasing, and $\theta/\delta\to\kappa_0$ because $\mathcal K_\vartheta'(0)=1$. Moreover $\theta(I)\subset U_\vartheta\subset(-x_\vartheta,x_\vartheta)$, so that $\varphi$ is well defined and continuous, with $\abs{\varphi}<\pi/2$ and $\varphi/\delta\to\rho\kappa_0$. By \cref{thm:Bangle} there is a unit-speed Frenet-regular Hahn curve $\gamma$ with turning angle $\theta$ and torsion angle $\varphi$; its invariants satisfy $\kappa(\wo)=\kappa_0$ and $\tau/\kappa=\cos\theta\sin\varphi/\sin\theta=\rho$ on $\Is$. Since $\mathcal K_\vartheta(f_\vartheta(\theta(t)))=q\kappa_0\delta(t)=\kappa_0\delta(\sigma(t))$ and $f_\vartheta(U_\vartheta)\subset U_\vartheta$, we have $\theta(\sigma(t))=f_\vartheta(\theta(t))$, i.e.\ \eqref{eq:H} holds with the sign $\varepsilon$; as $\abs{\theta}<2\vartheta$, \cref{thm:D} shows that $\gamma$ is a generalized Hahn helix with angle $\vartheta$ and an axis $u$ such that $\sgn(B\cdot u)=\varepsilon$. Finally, $\delta\circ h=(\tilde\kappa_0/\kappa_0)\,\delta$, so that by \cref{prop:reparam} the curve $\gamma\circ h$ on $h^{-1}(I)$ has the turning angle $\theta\circ h=\mathcal K_\vartheta^{-1}\circ(\tilde\kappa_0\delta)$, the torsion angle $\varphi\circ h=\arcsin(\rho\tan(\theta\circ h))$ and the Hahn speed $\tilde\kappa_0/\kappa_0$; the last claim follows from the uniqueness part of \cref{thm:Bangle}.

(c) Suppose that $\gamma$ has constant curvature $\kappa$ and constant torsion $\tau$ on $\Is$ and is a generalized Hahn helix with angle $\vartheta\in(0,\pi/2)$ and axis $u$ on a $\sigma$-invariant neighbourhood of $\wo$. As in (a), \eqref{eq:H} then holds, with a fixed sign $\varepsilon$, on $I_0^*$ for a smaller such neighbourhood $I_0$; the values $d:=\delta(t)$, $t\in I_0^*$, contain all sufficiently small $d$ of at least one sign. By~\eqref{eq:invariants}, $\sin\theta(t)=\kappa d$, $\sin\theta(\sigma(t))=q\kappa d$ and $\sin\varphi(t)=\tau d/\cos\theta(t)$, with $\abs{\theta},\abs{\varphi}<\pi/2$ for small $\abs{d}$. Hence
\[
  \tan\tfrac{\theta(t)}{2}=\tfrac{\kappa d}{2}\bigl(1+\tfrac{\kappa^2d^2}{4}\bigr),\enspace
  \tan\tfrac{\theta(\sigma(t))}{2}=\tfrac{q\kappa d}{2}\bigl(1+\tfrac{q^2\kappa^2d^2}{4}\bigr),\enspace
  \sin\varphi(t)=\tau d\bigl(1+\tfrac{\kappa^2d^2}{2}\bigr)
\]
up to $O(d^5)$, and $\cos\varphi(t)=1-\frac{\tau^2d^2}{2}$ and $\bigl(1-a^2\sec^2\frac{\theta(t)}{2}\bigr)^{1/2}=\sin\vartheta\bigl(1-\frac{\kappa^2d^2}{8}\cot^2\vartheta\bigr)$ up to $O(d^4)$, where $a=\cos\vartheta$. Inserting these expansions into~\eqref{eq:H}, the terms of order $d$ give $a\kappa(1+q)/2=\varepsilon\tau\sin\vartheta$ (in particular $\tau\neq0$), and then the difference of the left- and right-hand sides of~\eqref{eq:H} equals
\[
  \frac{a\kappa^3(1+q)}{16}\Bigl(2\bigl(q^2-q-1\bigr)-q\cot^2\vartheta\Bigr)d^3+O(d^5).
\]
Since $q^2-q-1<0$, the coefficient of $d^3$ does not vanish, which contradicts~\eqref{eq:H} for all small $d$ of that sign.
\end{proof}

For uniform steps the situation is different: if $\delta$ is constant, constant $\kappa$ and $\tau$ mean constant angles, \eqref{eq:H} reduces to a single equation for $\vartheta$, and polygons with equal edges and constant turning and non-zero torsion angles are indeed general helices. \Cref{prop:constratio}(c) thus reflects the variation of $\delta$ along the orbits. By (a), which shows that the ratio determines $\vartheta$ and $\varepsilon$, and by (b), near $\wo$ the unit-speed generalized Hahn helices with a given constant ratio $\tau/\kappa$ are unique up to similarities and linear changes of parameter about $\wo$, whereas the classical general helices with a given ratio correspond to arbitrary curvature functions.

\begin{remark}\label{rem:polygonhelices}
Parts (i)--(iii) of \cref{lem:helix} only use the rotation structure~\eqref{eq:rotation} along a single orbit, so they apply equally to the discrete Frenet frames of an arbitrary polygon. For polygons, general helices have been characterized in~\cite{YoonLeeLee2026} by a conserved Frenet-frame vector with free coefficients, by an intrinsic relation between three consecutive turning angles and two consecutive torsion angles that does not involve the helix angle (on the generic branch, where no torsion angle equals $0$ or $\pi$), and by a phase formulation. In the present notation the phase is the polar angle $\psi$ of $(b,c)=\sin\vartheta\,(\cos\psi,\sin\psi)$, and~\eqref{eq:rotbc} reads $\psi(\sigma(t))\equiv\pi-\psi(t)-\varphi(t)\pmod{2\pi}$, which is their phase relation. For polygons the sign of $\sin\psi$, i.e.\ of $B\cdot u$, may change from vertex to vertex; then the turning angles and the helix angle do not determine the helix, and relation~\eqref{eq:H} with a fixed sign $\varepsilon$ does not characterize helices: for given $\theta(t)$ and $\theta(\sigma(t))$ with $\abs{\theta(t)},\abs{\theta(\sigma(t))}<2\vartheta$ it has exactly two solutions $\varphi(t)\in(-\pi,\pi]$, which differ in the sign of $\hat c(t)$ in the proof of \cref{thm:D}, and only one of them also satisfies the second relation in \cref{lem:helix}(iii). In the Hahn setting the sign is fixed by the continuity of the frame on $I$: $\hat c$ is continuous and non-vanishing on each component of $\Is$ and tends to $\varepsilon\sin\vartheta$ at $\wo$ (continuity at $\wo$ alone fixes it only on a tail of each orbit). Continuity at $\wo$ also forces the axes obtained on the different orbits to coincide. These two facts, the resulting uniqueness in \cref{cor:allhelices}, the $q$-deformed Lancret ratio and \cref{prop:constratio} are the content of this section that is specific to the Hahn setting.
\end{remark}

\section{Examples}\label{sec:examples}

Unless stated otherwise, $q=\frac12$ and $\omega=1$, so that $\wo=2$ and $\delta(t)=1-t/2$; when $q$ varies in \cref{ex:sampled,ex:hahnhelix}, we take $\omega=2(1-q)$, so that $\wo=2$ there. All quantities were computed in $60$-digit arithmetic, the Hahn invariants directly from \cref{def:frame} by difference quotients, i.e.\ independently of the formulas they are compared with; infinite products and Jackson--N\"orlund sums were truncated when $q^k\abs{\delta(t)}<10^{-55}$.

\begin{example}\label{ex:circle}
Let $T(t)=(\cos t,\sin t,0)$ on $I=[0,4]$ and $\gamma(t):=p+\int_2^tT(x)\,d_{q,\omega}x$ with $p=(\sin2,-\cos2,0)$, so that $\gamma(2)=\Gamma(2)$ for the classical circle $\Gamma(t)=(\sin t,-\cos t,0)$, which satisfies $\Gamma'=T$. Since $T(\sigma(t))$ is obtained from $T(t)$ by the rotation through $\sigma(t)-t=\delta(t)$ about $e_3$, we have
\begin{gather*}
  \theta(t)=\delta(t)\in[-1,1],\qquad \varphi\equiv0,\qquad B\equiv e_3,\qquad
  \kappa(t)=\frac{\sin\delta(t)}{\delta(t)},\qquad \beta(t)=\frac{\cos\delta(t)-1}{\delta(t)},\qquad \tau\equiv0 .
\end{gather*}
The curve is Frenet-regular, and $\delta^2\kappa^2=\sin^2\delta<1$. \Cref{fig:circle} shows that $\gamma$ is not the circle $\Gamma$: its Hahn tangent $T(t)$ is the direction of the chords of the orbit polygons, not the classical tangent of $\gamma$. In accordance with \cref{prop:tangentfield}, $\kappa(2)=\kappa_\Gamma=1$. The series~\eqref{eq:JN} defining $\gamma$ may be differentiated term by term, since the $k$-th term and each of its derivatives are $O(q^k)$ uniformly on $I$; hence $\gamma$ is a $C^\infty$ curve, and $\gamma'(2)=T(2)$ by \cref{lem:JN}(b), so that its classical speed at $\wo$ is $1$. Differentiating the series twice gives $\gamma''(2)=\frac{2}{\qn{2}}T'(2)$, so that $\gamma'\times\gamma''\neq0$ near $\wo$. Therefore \cref{cor:fixedpoint}, applied to $\gamma$ on a small interval around $\wo$, predicts that the \emph{classical} curvature of $\gamma$ at $\wo$ equals $\kappa(2)/(\qn{2}/2)=4/3$; numerically we obtain $1.33333333333$.
\end{example}

\begin{figure}[H]
\centering
\includegraphics[width=\textwidth]{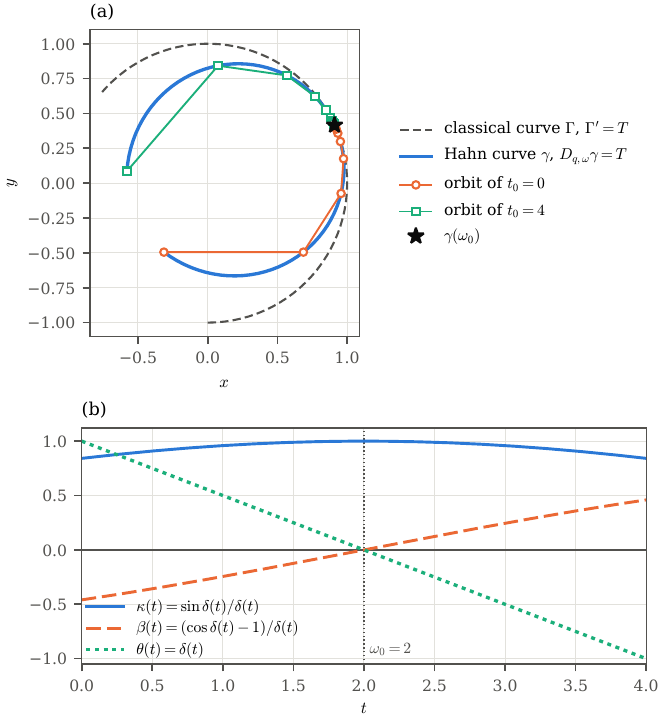}
\caption{\Cref{ex:circle} ($q=\frac12$, $\omega=1$, $\wo=2$). (a) The Hahn curve $\gamma$ on $I=[0,4]$, the classical circle $\Gamma$ with $\Gamma'=T$, and the orbit polygons of $t_0=0$ and $t_0=4$, whose vertices converge to $\gamma(\wo)$; the edges are parallel to $T(\sigma^n(t_0))$. (b) The turning angle $\theta=\delta$ and the invariants $\kappa$ and $\beta$}
\label{fig:circle}
\end{figure}

\begin{example}\label{ex:reconstruction}
Let $\kappa(t)=\frac45+\frac1{10}\sin t$ and $\tau(t)=\frac12\cos\frac t2$ on $I=[0,4]$. Here $\max_I\delta^2(\kappa^2+\tau^2)=0.89<1$, attained at $t=0$, so \cref{thm:B} applies; we take $p=0$ and $F_0=I$. The frame was computed as the product $\Pi(t)^{\top}F_0$ and the curve by the Jackson--N\"orlund sum. Recomputing the Hahn curvature and torsion of the resulting curve from \cref{def:frame} at the ten points $t\in\{0,0.5,1,1.5,1.9,2.1,2.5,3,3.5,4\}$ reproduces $\kappa$ and $\tau$ with errors below $10^{-59}$; see \cref{fig:reconstruction}.
\end{example}

\begin{figure}[H]
\centering
\includegraphics[width=\textwidth]{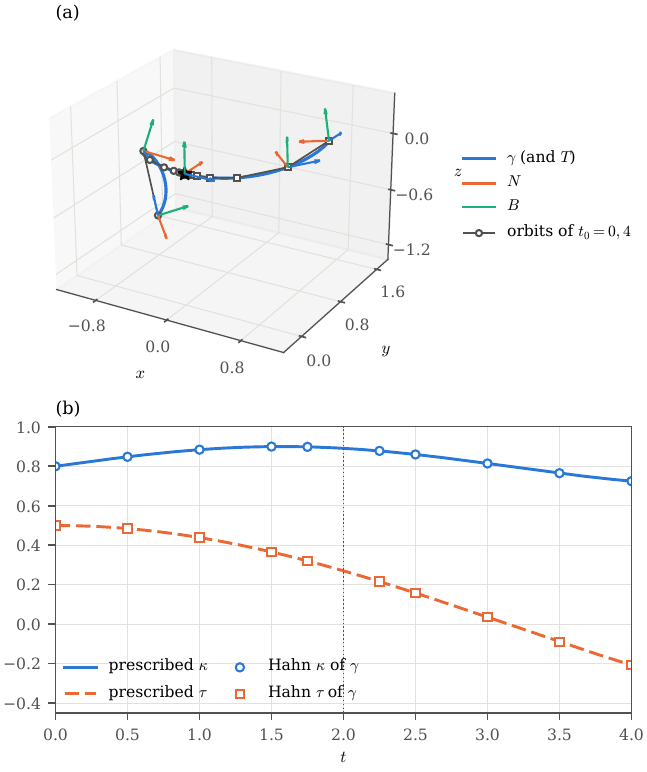}
\caption{\Cref{ex:reconstruction}. (a) The curve reconstructed from $\kappa(t)=\frac45+\frac1{10}\sin t$ and $\tau(t)=\frac12\cos\frac t2$ by \cref{thm:B}, with Hahn frames at $t\in\{0,1,2,3,4\}$ and the orbit polygons of $t_0=0$ and $t_0=4$. (b) The prescribed invariants (lines) and the Hahn invariants recomputed from the curve (markers)}
\label{fig:reconstruction}
\end{figure}

\begin{example}\label{ex:sampled}
(a) Let $\gamma(t)=(\cos t,\sin t,t)$. Then $s_{\cl}=\sqrt2$ and $\kappa_{\cl}=\tau_{\cl}=\frac12$, so $s_{\cl}\kappa_{\cl}=s_{\cl}\tau_{\cl}=\sqrt2/2\approx0.707107$, whereas \cref{cor:fixedpoint} predicts
\[
  \kappa_q(2)=\tfrac34\cdot\tfrac{\sqrt2}{2}=\tfrac{3\sqrt2}{8}\approx0.530330,\qquad \tau_q(2)=\tfrac{7}{12}\cdot\tfrac{\sqrt2}{2}=\tfrac{7\sqrt2}{24}\approx0.412479 .
\]
\Cref{tab:fixedpoint} and \cref{fig:limits}(a) confirm these limits; \cref{tab:classical} shows the classical limit at $t=0$. The reparametrizations $t\mapsto t+t_1$ and $t\mapsto-t$ of the helix are realized by proper rigid motions (a screw motion and the rotation $\operatorname{diag}(1,-1,-1)$). Since the Hahn quantities at $t$ depend only on the points $\gamma(\sigma^k(t))$, $0\le k\le3$, it follows that $\kappa_q(t)$ and $\tau_q(t)$ are even functions of $\delta(t)$ and $\beta_q(t)$ is an odd one; numerically, $\kappa_q(4-t)=\kappa_q(t)$, $\tau_q(4-t)=\tau_q(t)$ and $\beta_q(4-t)=-\beta_q(t)$ hold to more than $55$ digits. Since $\kappa_q$ and $\tau_q$ are moreover real-analytic functions of $t$ near $\wo$ (they are built from divided differences of the analytic curve $\gamma$), the first-order terms in \cref{thm:C} cancel for this curve, and the convergence at the fixed point is of second order (\cref{tab:fixedpoint}).

(b) For the twisted cubic $\gamma(t)=(t,t^2,t^3)$, which has no symmetry about $\wo=2$, \cref{fig:limits}(b) shows the deviation from the limit values of \cref{cor:fixedpoint} as $t\to\wo$ (with $q=\frac12$) and the classical limit at the point $t=\frac12$ as $q\to1$; both are of first order in $\abs{\delta}$. Moreover, $\bigl(\kappa_q(t)-\frac{\qn{2}}{2}(s_{\cl}\kappa_{\cl})(t)\bigr)/\delta(t)\to-0.1136$ and $\bigl(\tau_q(t)-\frac{\qn{3}}{3}(s_{\cl}\tau_{\cl})(t)\bigr)/\delta(t)\to-0.1023$ as $t\to2$, so the order in \cref{thm:C} is sharp in general.
\end{example}

\begin{table}[H]
\centering
\caption{Hahn invariants of the helix $(\cos t,\sin t,t)$ for $q=\frac12$, $\wo=2$ (\cref{ex:sampled}). At $4-t$ the values of $\kappa_q$ and $\tau_q$ are the same and $\beta_q$ changes sign. The last row gives the limits of \cref{cor:fixedpoint}; the classical values are $s_{\cl}\kappa_{\cl}=s_{\cl}\tau_{\cl}=0.707107$}
\label{tab:fixedpoint}
\begin{tabular}{ccccc}
\toprule
$t$ & $\delta(t)$ & $\kappa_q(t)$ & $\tau_q(t)$ & $\beta_q(t)$\\
\midrule
$0$ & $1$ & $0.494366$ & $0.356596$ & $-0.130746$ \\
$1$ & $0.5$ & $0.521006$ & $0.397843$ & $-0.0690540$ \\
$1.5$ & $0.25$ & $0.527978$ & $0.408778$ & $-0.0349982$ \\
$1.75$ & $0.125$ & $0.529741$ & $0.411551$ & $-0.0175583$ \\
$1.875$ & $0.0625$ & $0.530183$ & $0.412247$ & $-0.00878659$ \\
$1.9375$ & $0.03125$ & $0.530293$ & $0.412421$ & $-0.00439422$ \\
\midrule
$\to 2$ & $\to 0$ & $0.530330$ & $0.412479$ & $0$ \\
\bottomrule
\end{tabular}
\end{table}

\begin{table}[H]
\centering
\caption{The classical limit for the helix $(\cos t,\sin t,t)$ at $t=0$, with $\wo=2$ fixed and $q\to1$ (\cref{cor:classical})}
\label{tab:classical}
\begin{tabular}{ccccc}
\toprule
$q$ & $\delta(0)$ & $\kappa_q(0)$ & $\tau_q(0)$ & $\beta_q(0)$\\
\midrule
$0.5$ & $1$ & $0.494366$ & $0.356596$ & $-0.130746$ \\
$0.9$ & $0.2$ & $0.668729$ & $0.632581$ & $-0.0449216$ \\
$0.99$ & $0.02$ & $0.703536$ & $0.699984$ & $-0.00494988$ \\
$0.999$ & $0.002$ & $0.706753$ & $0.706399$ & $-4.99500\times 10^{-4}$ \\
$0.9999$ & $0.0002$ & $0.707071$ & $0.707036$ & $-4.99950\times 10^{-5}$ \\
\midrule
$\to 1$ & $\to 0$ & $0.707107$ & $0.707107$ & $0$ \\
\bottomrule
\end{tabular}
\end{table}

\begin{figure}[H]
\centering
\includegraphics[width=\textwidth]{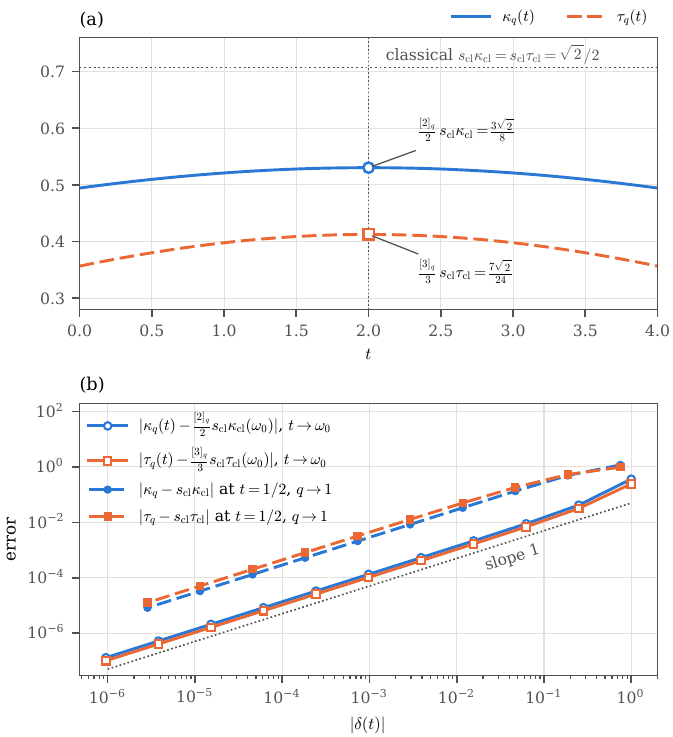}
\caption{\Cref{ex:sampled}. (a) Hahn curvature and torsion of the helix $(\cos t,\sin t,t)$ for $q=\frac12$; at $\wo=2$ they tend to the $q$-deformed values of \cref{cor:fixedpoint} and not to the classical value $\sqrt2/2$. (b) For the twisted cubic $(t,t^2,t^3)$: the deviation of $\kappa_q(t)$ and $\tau_q(t)$ from their limit values $\frac{\qn{2}}{2}(s_{\cl}\kappa_{\cl})(\wo)$ and $\frac{\qn{3}}{3}(s_{\cl}\tau_{\cl})(\wo)$ as $t\to\wo=2$ ($q=\frac12$), and their deviation from $s_{\cl}\kappa_{\cl}$ and $s_{\cl}\tau_{\cl}$ at $t=\frac12$ as $q\to1$ ($\wo=2$), plotted against $\abs{\delta(t)}$; the dotted line has slope $1$}
\label{fig:limits}
\end{figure}

\begin{example}\label{ex:hahnhelix}
Let $T(t)=(\sin\vartheta\cos t,\sin\vartheta\sin t,\cos\vartheta)$ with $\cot\vartheta=\frac12$, on $I=[0,4]$, and let $\gamma$ be the corresponding unit-speed Hahn curve. Since $T\cdot e_3=\cos\vartheta$, $\gamma$ is a generalized Hahn helix with axis $e_3$; here $\varepsilon=+1$, because $B(\wo)$ is the classical binormal of the helix $\Gamma$ with $\Gamma'=T$ and $B(\wo)\cdot e_3=\sin\vartheta>0$. On $I$ we have $\abs{\theta}\le0.886353<2\vartheta\approx2.21430$. \Cref{tab:hahnhelix} lists the angles and invariants; the residual of~\eqref{eq:H} is below $10^{-60}$ at all tabulated points. The ratio $\tau/\kappa$ is not constant: it increases from $0.282361$ at $t=0$ to $\frac{\qn{2}}{2}\cot\vartheta=0.375$ at $\wo$ (\cref{cor:lancretratio}) and stays below the classical Lancret ratio $\cot\vartheta=\frac12$. For the edge-normalized invariants~\eqref{eq:edgeinv}, $\hat\tau/\hat\kappa$ tends at $\wo$ to $\frac{3\qn{2}^2}{4\qn{3}}\cot\vartheta=\frac{27}{56}\approx0.482$, again below $\frac12$ (\cref{cor:lancretratio}); by \cref{rem:helixremarks}(a), this is also the ratio $\tau_{\cl}/\kappa_{\cl}$ of $\gamma$ at $\wo$. By \cref{prop:tangentfield}, $\kappa(2)=\sin\vartheta=2/\sqrt5$ and $\tau(2)=\frac34\cos\vartheta=3/(4\sqrt5)$, in agreement with the last row of \cref{tab:hahnhelix}. \Cref{fig:helix}(b) shows $\tau/\kappa$ for $q\in\{0.3,0.5,0.8\}$, with limits $(1+q)/4$ at $\wo$.

We also tested \cref{cor:allhelices}: starting from the turning angle $\theta(t)=1.1\sin\delta(t)$ (and $\vartheta$, $\varepsilon$ as above), we computed $\varphi$ by~\eqref{eq:phihelix} and the curve by \cref{thm:Bangle} with $F_0=I$; the resulting curve satisfies $T\cdot u=\cos\vartheta$, with $u=a\,e_1+\sin\vartheta\,e_3$, up to an error below $10^{-56}$ at seven points of $I$.
\end{example}

\begin{table}[H]
\centering
\caption{Turning and torsion angles and Hahn invariants of the generalized Hahn helix of \cref{ex:hahnhelix} ($q=\frac12$, $\cot\vartheta=\frac12$). At $4-t$ the angles change sign and $\kappa$, $\tau$ are unchanged. The last row gives the values at $\wo=2$ (\cref{prop:tangentfield,cor:lancretratio})}
\label{tab:hahnhelix}
\begin{tabular}{cccccc}
\toprule
$t$ & $\theta(t)$ & $\varphi(t)$ & $\kappa(t)$ & $\tau(t)$ & $\tau(t)/\kappa(t)$\\
\midrule
$0$ & $0.886353$ & $0.353320$ & $0.774771$ & $0.218765$ & $0.282361$ \\
$1$ & $0.446264$ & $0.169836$ & $0.863196$ & $0.304935$ & $0.353263$ \\
$1.5$ & $0.223490$ & $0.0841156$ & $0.886536$ & $0.327708$ & $0.369650$ \\
$1.75$ & $0.111789$ & $0.0419591$ & $0.892449$ & $0.333479$ & $0.373668$ \\
\midrule
$2$ & $0$ & $0$ & $0.894427$ & $0.335410$ & $0.375000$ \\
\bottomrule
\end{tabular}
\end{table}

\begin{figure}[H]
\centering
\includegraphics[width=\textwidth]{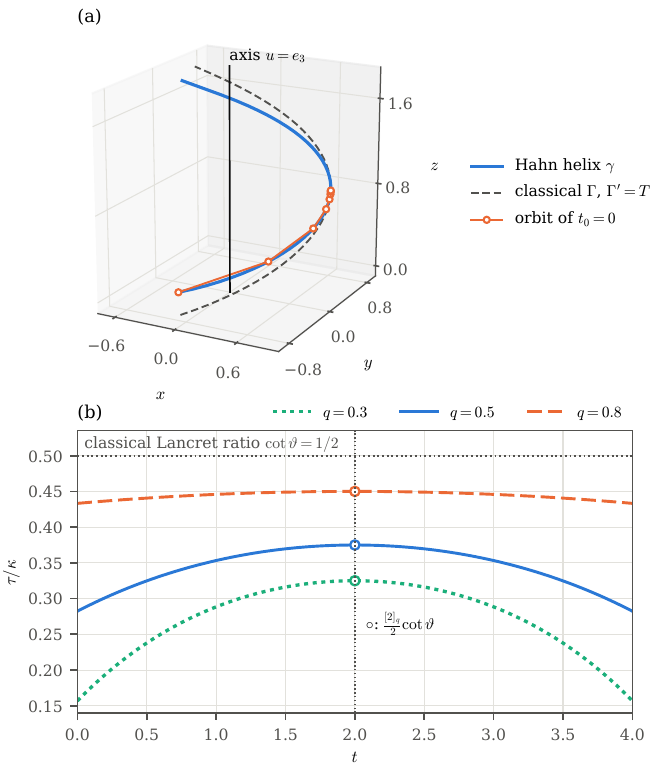}
\caption{\Cref{ex:hahnhelix}. (a) The generalized Hahn helix $\gamma$, the classical helix $\Gamma$ with $\Gamma'=T$, the axis $e_3$ and the orbit polygon of $t_0=0$. (b) The ratio $\tau/\kappa$ for $q\in\{0.3,0.5,0.8\}$ (with $\wo=2$); the circles mark the values $\frac{\qn{2}}{2}\cot\vartheta=(1+q)/4$ of \cref{cor:lancretratio}, and the dotted line the classical Lancret ratio $\cot\vartheta=\frac12$}
\label{fig:helix}
\end{figure}

\begin{example}\label{ex:relaxation}
In this example $\omega=0$, so that $\sigma(u)=qu$ and $\wo=0$ (Jackson's $q$-calculus), and $q$ is determined by the dynamics and the sampling step. Let $p$ be an equilibrium of an analytic system $\dot\xi=f(\xi)$ on an open subset of $\R^3$, and suppose that $f'(p)$ has the eigenvalues $-\lambda_1,-\lambda_2,-\lambda_3$ with $0<\lambda_1<\lambda_2<\lambda_3$ and an eigenvector $v_1$ for $-\lambda_1$, and that there are no resonances, i.e.\ $\lambda_i\neq m_1\lambda_1+m_2\lambda_2+m_3\lambda_3$ for all integers $m_j\ge0$ with $m_1+m_2+m_3\ge2$. Since the eigenvalues lie in the Poincar\'e domain, Poincar\'e's theorem~\cite[Chapter~5]{Arnold1988} yields a ball $V$ about $0$ and a real-analytic diffeomorphism $\mathcal L$ of $V$ onto a neighbourhood of $p$ with $\mathcal L(0)=p$ and $\mathcal L'(0)e_1=v_1$ that maps the solutions of the linear system $\dot\zeta_i=-\lambda_i\zeta_i$ in $V$ to solutions of $\dot\xi=f(\xi)$. Let $\xi$ be a solution with $\xi(t)\to p$ as $t\to\infty$. For large $t$, $\xi(t)$ lies in $\mathcal L(V)$, and then $\xi(t)=\mathcal L(\zeta(t))$ with $\zeta_i(t)=a_ie^{-\lambda_it}$ for constants $a_i$; we assume $a_1\neq0$, which means that $\xi(t)$ converges to $p$ tangentially to $v_1$. With $u:=\zeta_1(t)=a_1e^{-\lambda_1t}$ we obtain $\xi(t)=\gamma(u)$, where
\[
  \gamma(u):=\mathcal L\bigl(u,\,b_2\abs{u}^{\lambda_2/\lambda_1},\,b_3\abs{u}^{\lambda_3/\lambda_1}\bigr),\qquad b_i:=a_i\abs{a_1}^{-\lambda_i/\lambda_1}.
\]
Here $\abs{u}$ is so small that the argument of $\mathcal L$ lies in $V$; for $u$ of the sign of $a_1$, $\gamma$ parametrizes the trajectory of $\xi$, and for the other sign that of the solution with $a_1$ replaced by $-a_1$. For a time step $\Delta t>0$ and $q:=e^{-\lambda_1\Delta t}$, the samples $P_n:=\xi(n\Delta t)$ therefore satisfy $P_n=\gamma(q^na_1)$ for all large $n$; from some index on, they are the vertices of an orbit polygon (\cref{rem:polygons}) of the Hahn curve $\gamma$. The time-$\Delta t$ map of the flow acts on the trajectory by $\gamma(u)\mapsto\gamma(qu)$. Since the signed arc length from $p$ to $\gamma(u)$ is $\abs{v_1}u+o(u)$, the Koenigs coordinate $\lim_{n\to\infty}q^{-n}\Sigma^n$ of the contraction $\Sigma$ that this map induces on the trajectory, written in terms of the arc length as in \cref{prop:coordfree}, is $\abs{v_1}u$. In particular, the Hahn parameter does not depend on $\Delta t$; only $q$ does.

Now let $\lambda_2>4\lambda_1$. Then the functions $u\mapsto\abs{u}^{\lambda_i/\lambda_1}$, $i=2,3$, are of class $C^4$ and have vanishing derivatives up to order four at $0$. Hence $\gamma$ is of class $C^4$ on an open interval around $0$, and it has the same $4$-jet at $0$ as the analytic invariant curve $W(u):=\mathcal L(u,0,0)$, which is tangent to $v_1$. Assume that the curvature $\kappa_W$ of $W$ at $p$ is non-zero, and let $\tau_W$ be its torsion at $p$. Then $\gamma$ satisfies the assumptions of the Setting of \cref{sec:limits} near $0$, with $\kappa_{\cl}(0)=\kappa_W$ and $\tau_{\cl}(0)=\tau_W$, and the edge lengths $\ell_n:=\abs{P_{n+1}-P_n}$ satisfy $\ell_n=\abs{v_1}(1-q)q^n\abs{a_1}\bigl(1+O(q^n)\bigr)$. With the angles $\theta_n$ and $\varphi_n$ of \cref{prop:polygons}, we obtain, with errors $O(q^n)$,
\begin{enumerate}[label=\textup{(\roman*)}]
\item $\ell_{n+1}/\ell_n\to q=e^{-\lambda_1\Delta t}$;
\item $\theta_n/\ell_n\to\frac{\qn{2}}{2}\,\kappa_W$ and $\varphi_n/\ell_n\to\frac{\qn{3}}{3}\,\tau_W$;
\item $2\theta_n/(\ell_n+\ell_{n+1})\to\kappa_W$ and $3\varphi_n/(\ell_n+\ell_{n+1}+\ell_{n+2})\to\tau_W$;
\item $\varphi_n/\theta_n\to\frac{2\qn{3}}{3\qn{2}}\,\tau_W/\kappa_W$.
\end{enumerate}
Indeed, (iii) is \cref{prop:polygons} for $x_n=q^na_1$ and $x^*=0$; (ii) follows from (i) and (iii), and (iv) from the formula for $\varphi_n/\theta_n$ in \cref{prop:polygons} and (i), since $(\ell_n+\ell_{n+1}+\ell_{n+2})/(\ell_n+\ell_{n+1})=\qn{3}/\qn{2}+O(q^n)$. By \cref{prop:reparam} and the relation between the angles stated before \cref{prop:polygons}, for either sign of $a_1$, $\sin\theta_n/\ell_n$ and $\cos\theta_n\sin\varphi_n/\ell_n$ are the quotients $\kappa/s$ and $\tau/s$ of the Hahn curve $\gamma$ at $q^na_1$, which depend only on the orbit polygons. They differ from the quotients in (ii) by factors $1+O(q^{2n})$, and the limits in (ii) are those given by \cref{cor:fixedpoint}, since the Hahn speed tends to $s_{\cl}(0)$. Thus a single sampled trajectory determines $\lambda_1$, $\kappa_W$ and $\tau_W$ through (i) and (iii), whereas the quotients by one edge length in (ii) and the ratio of the angles in (iv), which involves no normalization at all, remain biased by the factors of \cref{cor:fixedpoint} and by their quotient, respectively, however many samples near $p$ are used; the bias disappears only as $\Delta t\to0$. Consistent estimators of $\kappa_W$ and $\tau_W$ are available without Hahn calculus, for instance those in (iii) or, for the curvature, the reciprocal radius of the circle through three consecutive vertices (see~\cite{Sullivan2008} and the discussion before~\eqref{eq:edgeinv}). What the Hahn setting adds is the structure of the samples, with the linearizing coordinate as Hahn parameter, and the exact asymptotic bias in (ii) and (iv). If the vertices carry errors of size $\epsilon$, then $\theta_n$ and $\varphi_n$ change by $O(\epsilon/\ell_n)$ and $O(\epsilon/(\theta_n\ell_n))$, so that the relative errors of the estimates in (iii) are of order $\epsilon/(\kappa_W\ell_n^2)$ and $\epsilon/(\kappa_W\tau_W\ell_n^3)$, whereas their deviations from the limits decay at least like $\ell_n$; with noisy data, only a window of indices $n$ is useful, and it is narrower for the torsion.

For the Lorenz system~\cite{Lorenz1963} with the parameters $10$, $\frac12$ and $\frac83$,
\begin{equation}\label{eq:lorenz}
  \dot x=10(y-x),\qquad \dot y=\tfrac12x-y-xz,\qquad \dot z=xy-\tfrac83z,
\end{equation}
the origin is a stable node with $\lambda_1=(11-\sqrt{101})/2\approx0.475062$, $\lambda_2=\frac83$, $\lambda_3=(11+\sqrt{101})/2\approx10.5249$ and $v_1=(1,c_1,0)$, $c_1:=1-\lambda_1/10$; moreover $\lambda_2/\lambda_1\approx5.61>4$. There are no resonances: since $\lambda_2$ and $\lambda_1+\lambda_3=11$ are rational and $\lambda_1$ is irrational, a relation $\lambda_i=m_1\lambda_1+m_2\lambda_2+m_3\lambda_3$ with integers $m_j\ge0$ forces $m_1-m_3=1,0,-1$ for $i=1,2,3$, respectively, and then, by comparison of the rational parts, $m_i=1$ and $m_j=0$ for $j\neq i$. Near $0$, $W$ is the graph of a function of $x$. The invariance equations determine its Taylor coefficients recursively, the coefficients of $x^m$ in $y$ and in $z$ having the denominators $\lambda_3-m\lambda_1$ and $\lambda_2-m\lambda_1$, respectively, which do not vanish because there are no resonances; they show that $y$ is an odd and $z$ an even function of $x$ along $W$, and they give $y=c_1x+c_3x^3+O(x^5)$ and $z=d_2x^2+O(x^4)$ with $d_2=c_1/(\lambda_2-2\lambda_1)$ and $c_3=-d_2/(\lambda_3-3\lambda_1)$. Hence
\begin{gather*}
  \kappa_W=\frac{2d_2}{1+c_1^2}\approx0.581877,\qquad \tau_W=-\frac{3c_3}{d_2(1+c_1^2)}\approx0.172856,\\
  \frac{\tau_W}{\kappa_W}\approx0.297067.
\end{gather*}
We computed the solution $\xi$ of~\eqref{eq:lorenz} with $\xi(0)=(2,-1,\frac32)$ by the Taylor series method of order $50$ with step $\frac1{40}$; halving the step changes the computed points by less than $10^{-55}$ relative to their distance from $0$. The computed solution converges to $0$ tangentially to $v_1$, so that $a_1\neq0$. We sampled $\xi$ with $\Delta t=1$ and $\Delta t=\frac14$, for which $q=e^{-\lambda_1}\approx0.621846$ and $q=e^{-\lambda_1/4}\approx0.888016$. \Cref{tab:lorenz} lists the quotients (i)--(iv) at the time $n\Delta t=30$, i.e.\ at $n=30$ for $\Delta t=1$ and at $n=120$ for $\Delta t=\frac14$. For $n\ge25$ ($\Delta t=1$) and $n\ge100$ ($\Delta t=\frac14$), all of them agree with their limits to nine significant digits, and at the tabulated $n$ the quotient $-\log(\ell_{n+1}/\ell_n)/\Delta t$ equals $\lambda_1$ to twelve digits. In fact, the deviations from the limits decay like $q^{2n}$. The rotation $(x,y,z)\mapsto(-x,-y,z)$ leaves~\eqref{eq:lorenz} and $W$ invariant and maps the polygon with vertices $W(q^ku)$ onto that with vertices $W(-q^ku)$, so that for $W$ the quotients are even analytic functions of $u$; compare \cref{ex:sampled}(a). The terms $b_i\abs{u}^{\lambda_i/\lambda_1}$ change them only by $O(q^{(\lambda_2/\lambda_1-3)n})$, and $\lambda_2/\lambda_1-3>2$. For $\Delta t=1$, the quotients in (ii) and (iv) underestimate $\kappa_W$, $\tau_W$ and $\tau_W/\kappa_W$ by $18.9\%$, $33.0\%$ and $17.4\%$, in accordance with $\frac{\qn{2}}{2}\approx0.810923$, $\frac{\qn{3}}{3}\approx0.669513$ and $\frac{2\qn{3}}{3\qn{2}}\approx0.825618$, and for $\Delta t=\frac14$ still by $5.6\%$, $10.8\%$ and $5.5\%$.
\end{example}

\begin{table}[H]
\centering
\caption{Quotients of the edge lengths and angles of the samples $P_n=\xi(n\Delta t)$ of the solution of~\eqref{eq:lorenz} in \cref{ex:relaxation} at the time $n\Delta t=30$, and their limits (i)--(iv), where $q=e^{-\lambda_1\Delta t}$, $\kappa_W=0.581877$ and $\tau_W=0.172856$. The percentages are the deviations from $\kappa_W$, $\tau_W$ and $\tau_W/\kappa_W$, respectively}
\label{tab:lorenz}
\begin{tabular}{llll}
\toprule
quotient & limit & $\Delta t=1$, $n=30$ & $\Delta t=\frac14$, $n=120$\\
\midrule
$\ell_{n+1}/\ell_n$ & $q=e^{-\lambda_1\Delta t}$ & $0.621846$ & $0.888016$ \\
$\theta_n/\ell_n$ & $\frac{\qn{2}}{2}\,\kappa_W$ & $0.471858$ ($-18.9\%$) & $0.549297$ ($-5.6\%$) \\
$2\theta_n/(\ell_n+\ell_{n+1})$ & $\kappa_W$ & $0.581877$ & $0.581877$ \\
$\varphi_n/\ell_n$ & $\frac{\qn{3}}{3}\,\tau_W$ & $0.115730$ ($-33.0\%$) & $0.154222$ ($-10.8\%$) \\
$3\varphi_n/(\ell_n+\ell_{n+1}+\ell_{n+2})$ & $\tau_W$ & $0.172856$ & $0.172856$ \\
$\varphi_n/\theta_n$ & $\frac{2\qn{3}}{3\qn{2}}\,\frac{\tau_W}{\kappa_W}$ & $0.245264$ ($-17.4\%$) & $0.280762$ ($-5.5\%$) \\
\bottomrule
\end{tabular}
\end{table}

\section{Concluding remarks}\label{sec:conclusion}

We have shown that the Frenet theory of curves in Hahn calculus combines two structures. Along each orbit of $\sigma$ it is the theory of discrete Frenet frames of the orbit polygons, with frames transported by explicit rotations (\cref{thm:A}). At the fixed point the structure becomes infinitesimal, and continuity there couples all orbits, which yields existence and uniqueness (\cref{thm:B,thm:Bangle}), with convergence to the classical curves as $q\to1$ (\cref{prop:qto1}), and a sharp planarity criterion (\cref{thm:planar}). The parametrization is part of this structure rather than a gauge freedom: for a regular Hahn curve that is a topological embedding, it is determined, up to a linear change of parameter about $\wo$, by the trace and the contraction that $\sigma$ induces on it (\cref{prop:reparam,prop:coordfree,rem:coordfree}). The interplay of the two structures produces phenomena without counterpart for uniform steps. At the fixed point, curvature and torsion carry the factors $\qn{2}/2$ and $\qn{3}/3$ relative to the classical invariants (\cref{thm:C}, \cref{cor:fixedpoint}); they are produced by the geometric decay of the edges and disappear under normalization by edge lengths, but their quotient persists in the limiting ratio of torsion and turning angles of every polygon inscribed in a smooth curve with monotonically converging vertices and edge lengths decaying with ratio $q$ (\cref{cor:edge,prop:polygons}). The Lancret ratio of generalized Hahn helices is $q$-deformed, also after edge normalization, and a constant ratio $\tau/\kappa$ does not characterize helices (\cref{thm:D,cor:lancretratio,prop:constratio}). Under the assumptions of \cref{ex:relaxation}, Hahn curves also arise when a trajectory relaxing towards a stable node is sampled at uniform time steps, and the factors $\qn{2}/2$ and $\qn{3}/3$ become measurable biases of discrete estimates of curvature and torsion.

Several directions remain open. (i) Rotation-minimizing (Bishop) frames~\cite{Bishop1975} have a well-developed discrete counterpart~\cite{WangJuttlerZhengLiu2008}; a Hahn version would remove the condition $T\times DT\neq0$ and should lead to Hahn analogues of slant helices and of Bertrand and Mannheim pairs. (ii) Surfaces can be approached with the two-parameter Hahn calculus~\cite{HamzaMakharesh2016} (cf.\ the partial differentiation on time scales in~\cite{BohnerGuseinov2004}); the Gauss--Codazzi equations should then appear as compatibility conditions for step rotations in two directions. (iii) For curves in $\R^n$, a Hahn Frenet frame $(E_1,\dots,E_n)$ is obtained by the Gram--Schmidt process from $D\gamma,\dots,D^{n-1}\gamma$, completed to a positively oriented frame, with Hahn curvatures $\kappa_j:=DE_j\cdot E_{j+1}$; for $n=3$ this gives $(T,N,B)$, $\kappa$ and $m_{23}$. Since the span of $E_1(t),\dots,E_j(t)$ is determined by the $j+1$ points $\gamma(\sigma^i(t))$, $0\le i\le j$, and should approximate the classical osculating $j$-space at the centroid of their parameters, as it does for $j=1,2$ by \cref{prop:frames}, the argument of \cref{rem:qnumbers} leads to the following conjecture: for a smooth curve whose derivatives $\gamma',\dots,\gamma^{(n-1)}$ are linearly independent, the $j$-th Hahn curvature tends at $\wo$ to $\frac{\qn{j+1}}{j+1}\,s_{\cl}(\wo)\kappa_{j,\cl}(\wo)$, where $\kappa_{j,\cl}$ is the $j$-th classical curvature, $1\le j\le n-1$. For $n=3$ this follows from \cref{cor:fixedpoint,prop:fixedpoint}. For $n=4$, numerical experiments in $160$-digit arithmetic agree with the conjecture: for the curves $(t,t^2/2,t^3/6,t^4/24+t^2/5)$ with $\wo=0.7$ and $(\sin t,\cos 2t,e^{t/2},t^3/3)$ with $\wo\in\{0.3,-1.2\}$, and for $q\in\{0.3,0.5,0.8\}$, the Hahn curvatures at $t=\wo\pm10^{-8}$ and $t=\wo\pm10^{-16}$ differ from the conjectured limits by less than $5\abs{t-\wo}$. (iv) \Cref{thm:A,thm:B} use only the product rule, the identity $f(\sigma(t))=f(t)+\delta(t)Df(t)$ and the contraction of $\sigma$ to its fixed point, so they should extend to the general quantum difference operators of~\cite{HamzaSarhanShehataAldwoah2015}; it would be interesting to determine the analogues of the factors $\qn{k}/k$. (v) Curve flows on Hahn grids, in the spirit of~\cite{InoguchiKajiwaraMatsuuraOhta2012}, may lead to integrable equations on geometric lattices. (vi) By \cref{prop:coordfree}, a contraction $\Phi$ of $\R^3$ with fixed point $p$ that leaves an arc through $p$ invariant, and satisfies the assumptions of that proposition on it, defines a Hahn curve whose orbit polygons are orbits of $\Phi$. Whether this curve is Frenet-regular at $\wo$ depends on $\Phi$. For a diagonal linear contraction with eigenvalues $1>\mu_1>\mu_2>\mu_3>0$ (on arcs tangent to $e_1$ this requires $q=\mu_1$), an invariant $C^2$ arc tangent to $e_1$ has non-zero curvature at $p=0$ only if $\mu_2=\mu_1^2$ or $\mu_3=\mu_1^2$, since along an orbit the ratios of the $e_2$- and $e_3$-components to the square of the $e_1$-component behave like $(\mu_2/\mu_1^2)^n$ and $(\mu_3/\mu_1^2)^n$. Non-linear terms change this: for $\Phi(x,y,z)=(\mu_1x,\mu_2y+ax^2,\mu_3z)$ with $a\neq0$ and $\mu_2,\mu_3\neq\mu_1^2$, every invariant $C^2$ arc tangent to $e_1$ has curvature $2\abs{a}/\abs{\mu_1^2-\mu_2}$ at $0$. In \cref{ex:relaxation}, $\Phi$ is the time-$\Delta t$ map of a flow near a stable node, with $\mu_i=e^{-\lambda_i\Delta t}$, and the conditions $\lambda_2>4\lambda_1$ and $\kappa_W\neq0$ make the induced Hahn curves Frenet-regular near $\wo$. It would be interesting to characterize the contractions that induce Frenet-regular Hahn curves, and to develop a Hahn Frenet theory that admits singularities at $\wo$. Relaxation towards stable equilibria produces such singularities. If $\lambda_2<2\lambda_1$ in \cref{ex:relaxation}, the trajectories with $b_2\neq0$ have curvature of order $\abs{u}^{\lambda_2/\lambda_1-2}$ near $p$. At a node with the resonance $\lambda_2=2\lambda_1$, the normal form $\dot\zeta_1=-\lambda_1\zeta_1$, $\dot\zeta_2=-2\lambda_1\zeta_2+c\zeta_1^2$ in the plane~\cite[Chapter~5]{Arnold1988} has, besides the $\zeta_2$-axis, the orbits $\zeta_2=\zeta_1^2\bigl(C_0-(c/\lambda_1)\ln\abs{\zeta_1}\bigr)$, $C_0\in\R$; if $c\neq0$, their curvature diverges like $(2\abs{c}/\lambda_1)\abs{\ln\abs{\zeta_1}}$ at the node. At a stable focus, finally, the trajectories are spirals without a tangent at the focus, whose curvature is of the order of the inverse distance to the focus, so that already the Hahn tangent is singular at $\wo$.

    \section*{Declarations}

        \begin{itemize}
        \item Funding: No funding was received for this study.
        \item Conflict of interest: The authors declare that there is no conflict of interest regarding the publication of this article.
        \item Ethics approval and consent to participate: Not applicable.
        \item Consent for publication: Not applicable.
        \item Data availability: No data were used or generated in the preparation of this article. All results are purely mathematical and self-contained within the paper. 
        \item Materials availability: Not applicable.
        \item Code availability: Not applicable.
        \item Author contribution: \textbf{Umut Selvi:} Conceptualization, Formal analysis, Methodology, Writing -- original draft, Writing -- review \& editing. \textbf{Derya Sağlam:} Conceptualization, Formal analysis, Methodology, Writing -- original draft, Writing -- review \& editing.
        \end{itemize}

\end{document}